\documentclass[%
  a4paper,
  onecolumn,
  colorlinks,
]{preprint}

\usepackage[noadjust]{cite}

\usepackage[english]{babel}

\usepackage{amsmath}
\usepackage{amssymb}
\usepackage{amsthm}
\usepackage{mathtools}
\usepackage{dsfont}

\usepackage[linesnumbered, lined, algoruled]{algorithm2e}

\usepackage{xcolor}
\usepackage{graphicx}

\usepackage{tikz}
\usepackage{pgfplots}
  \pgfplotsset{%
    compat        = 1.17,
    table/col sep = comma}
  \usepgfplotslibrary{external}
  \tikzset{external/system call = {%
    pdflatex \tikzexternalcheckshellescape
      -halt-on-error
      -interaction=batchmode
      -jobname "\image" "\texsource"}}
  \tikzexternaldisable

\usepackage{caption}
\usepackage{subcaption}

\newcommand{\mbf}[1]{\boldsymbol{#1}}

\newcommand*{\bmat}[1]{\begin{bmatrix}#1\end{bmatrix}}
\newcommand*{\sbmat}[1]{\bigl[ \begin{smallmatrix}#1\end{smallmatrix} \bigr ]}

\newcommand{\dataSet}{\ensuremath{\mathbb{D}}}
\newcommand{\C}{\ensuremath{\mathbb{C}}}
\newcommand{\R}{\ensuremath{\mathbb{R}}}
\newcommand{\N}{\ensuremath{\mathbb{N}}}

\newcommand{\trans}{\ensuremath{\mkern-1.5mu\mathsf{T}}}
\DeclareMathOperator{\mdiag}{diag}
\DeclareMathOperator*{\argmax}{argmax}
\DeclareMathOperator*{\argmin}{argmin}

\newcommand{\bA}{\ensuremath{\boldsymbol{A}}}
\newcommand{\bC}{\ensuremath{\boldsymbol{C}}}
\newcommand{\bE}{\ensuremath{\boldsymbol{E}}}

\newcommand{\bb}{\ensuremath{\boldsymbol{b}}}
\newcommand{\bc}{\ensuremath{\boldsymbol{c}}}
\newcommand{\be}{\ensuremath{\boldsymbol{e}}}
\newcommand{\bv}{\ensuremath{\boldsymbol{v}}}
\newcommand{\bx}{\ensuremath{\boldsymbol{x}}}
\newcommand{\bz}{\ensuremath{\boldsymbol{z}}}

\newcommand*{\w}{\ensuremath{\boldsymbol{w}}}
\newcommand{\nBasis}{\ensuremath{\boldsymbol{p}}}
\newcommand{\dBasis}{\ensuremath{\boldsymbol{q}}}
\newcommand{\AAAiter}[1][k]{^{(#1)}}
\newcommand{\skwfIter}[1][p]{_{(#1)}}
\newcommand{\LowMat}{\left(\FunValsMat\CauchyMat - \CauchyMat\interpFunVals\right)}
\newcommand{\WFCoefMat}{\boldsymbol{F}\skwfIter[p-1]}
\newcommand{\SKWFWeightMat}{\mbf D\skwfIter[p-1]}
\newcommand{\WFRHS}{\mbf b\skwfIter[p-1]}
\newcommand{\CauchyMat}{\bC}
\newcommand{\interpFunVals}{\boldsymbol{H}}
\newcommand{\FunValsMat}{\mbf G}

\newcommand{\zeros}{\ensuremath{\boldsymbol{0}}}

\theoremstyle{plain}\newtheorem{theorem}{Theorem}
\theoremstyle{plain}\newtheorem{lemma}{Lemma}
\theoremstyle{plain}
\theoremstyle{definition}\newtheorem{remark}{Remark}

\newcommand{\diff}{\ensuremath{\operatorname{d}}}
\newcommand{\drdw}{\frac{\partial r(z;\w)}{\partial \w}}
\newcommand{\dLevydw}{\frac{\operatorname{d}\ELevy}{\operatorname{d}\w}}
\newcommand{\dSKdw}{\frac{\operatorname{d}\Esk}{\operatorname{d}\w}}
\newcommand{\dWFdw}{\frac{\operatorname{d}\Ewf}{\operatorname{d}\w}}
\newcommand{\dEdw}{\frac{\operatorname{d}\ELtwo}{\operatorname{d}\w}}

\newcommand{\triWave}{\textrm{triWave}}

\newcommand{\relu}{\textrm{relu}}

\newcommand{\maxOrder}{\ensuremath{k_{\max}}}
\newcommand{\maxOrderSK}{\ensuremath{p_{\max}}}
\newcommand{\imunit}{\ensuremath{\mathfrak{i}}}

\newcommand{\ltwo}{\ensuremath{\ell_{2}}}
\newcommand{\linf}{\ensuremath{\ell_{\infty}}}

\newcommand{\AAA}{\texttt{AAA}}
\newcommand{\SK}{\texttt{SK}}
\newcommand{\WF}{\texttt{WF}}

\newcommand{\AAAA}{\texttt{NL-AAA}}

\newcommand{\ELtwo}{\ensuremath{E}}
\newcommand{\ELevy}{\ensuremath{E_{\mathtt{Levy}}}}
\newcommand{\Esk}{\ensuremath{E_{\mathtt{SK}}}}
\newcommand{\Ewf}{\ensuremath{E_{\mathtt{WF}}}}

\usepackage{todonotes}

\newcommand{\plotfontsize}{\small}

\definecolor{matlabblue}{HTML}{0072BD}
\definecolor{matlaborange}{HTML}{D95319}
\definecolor{matlabyellow}{HTML}{EDB120}
\definecolor{matlabpurple}{HTML}{7E2F8E}
\definecolor{matlabgreen}{HTML}{77AC30}
\definecolor{matlablightblue}{HTML}{4DBEEE}
\definecolor{matlabred}{HTML}{A2142F}

\colorlet{trueDataColor}{black}
\colorlet{AAAColor}{matlabblue}
\colorlet{AAAAColor}{matlaborange}

\tikzstyle{trueData} = [
  trueDataColor,
  solid,
  only marks,
  mark size = 1.75pt,
  mark      = *
]

\tikzstyle{AAAConverge} = [
  AAAColor,
  line width   = 1.25pt,
  mark         = square*,
  mark size    = 1.75pt,
  mark options = {solid},
  mark repeat  = {2}
]

\tikzstyle{AAAAConverge} = [
  AAAAColor,
  line width   = 1.25pt,
  mark         = triangle*,
  mark size    = 1.75pt,
  mark options = {solid},
  mark repeat  = {2}
]

\tikzstyle{AAAResponse} = [
  AAAConverge,
  dashdotted,
  mark repeat = {50},
  mark phase  = {0}
]

\tikzstyle{AAAAResponse} = [
  AAAAConverge,
  dashed,
  mark repeat = {50},
  mark phase  = {25}
]

\begin{document}


\title{A refined nonlinear least-squares method for the rational approximation
problem}

\author[$\ast$]{Michael S. Ackermann}
\affil[$\ast$]{Department of Mathematics, Virginia Tech, Blacksburg,
  VA 24061, USA.\authorcr
  \email{amike98@vt.edu}, \orcid{0000-0003-3581-6299}}

\author[$\ast$]{Linus Balicki}
\affil[$\dagger$]{
  Department of Mathematics, Virginia Tech, Blacksburg,
  VA 24061, USA.\authorcr
  \email{balicki@vt.edu}, \orcid{0000-0002-8901-2889}
}

\author[$\ddagger$]{\authorcr{Serkan Gugercin}}
\affil[$\ddagger$]{%
  Department of Mathematics and Division of Computational Modeling and Data
  Analytics, Academy of Data Science, Virginia Tech,
  Blacksburg, VA 24061, USA.\authorcr
  \email{gugercin@vt.edu}, \orcid{0000-0003-4564-5999}
}

\author[$\S$]{Steffen W. R. Werner}
\affil[$\S$]{%
  Department of Mathematics, Division of Computational Modeling and
  Data Analytics, and National Security Institute, Virginia Tech,
  Blacksburg, VA 24061, USA.\authorcr
  \email{steffen.werner@vt.edu}, \orcid{0000-0003-1667-4862}
}

\shorttitle{Nonlinear least-squares AAA algorithm}
\shortauthor{Ackermann, Balicki, Gugercin, Werner}
\shortdate{2026-01-27}
\shortinstitute{}



\abstract{%
The adaptive Antoulas-Anderson (AAA) algorithm for rational approximation is a
widely used method for the efficient construction of highly accurate rational
approximations to given data.
While AAA can often produce rational approximations accurate to any
prescribed tolerance, these approximations may have degrees larger than
what is actually required to meet the given tolerance.
In this work, we consider the adaptive construction of interpolating rational
approximations while aiming for the smallest feasible degree to satisfy a
given error tolerance.
To this end, we introduce refinement approaches to the linear least-squares
step of the classical AAA algorithm that aim to minimize the true nonlinear
least-squares error with respect to the given data.
Furthermore, we theoretically analyze the derived approaches in terms of the
corresponding gradients from the resulting minimization problems and use these insights
to propose a new greedy framework that ensures monotonic error convergence.
Numerical examples from function approximation and model order reduction verify
the effectiveness of the proposed algorithm to construct accurate rational
approximations of small degrees.
}

\novelty{%
We develop a new greedy-type adaptive approximation method for the
construction of rational approximations from given data.
The proposed method selects in every step a suitable interpolation point
to minimize an error measure and efficiently solves a nonlinear least-squares
problem on the rest of the data.
We provide a theoretical analysis of the minimization approaches in terms of
their gradients.
}

\maketitle


\section{Introduction}%
\label{sec:intro}

Approximating complex functions by rational functions is an important
task in the computational sciences.
A recent popular approach for this purpose is the adaptive Antoulas-Anderson (\AAA)
algorithm~\cite{NakST18}, which had wide ranging success in many fields due to
its ability to achieve highly accurate approximations quickly;
see~\cite{NakT25} for a survey of the method.
In certain applications, it is desired to construct highly accurate
rational functions that have a low degree.
One of these applications is model order reduction in the frequency
domain~\cite{Ant05, AntBG20, BenSGetal21}.
Here, the input-to-output behavior of a dynamical system can be described via
a rational function, the so-called transfer function.
The degree of that rational function is also the order of the corresponding
dynamical system.
The task in model order reduction is then to find a rational approximation to
the transfer function that is accurate in regions of interest while having the
smallest possible degree to ensure fast evaluations.

Data-driven reduced-order modeling has become common practice. 
In this paper, we focus on data-driven reduced-order modeling  
grounded in rational approximation theory. 
For an incomplete list of such methods,
see, e.g.,~\cite{AntA86, MayA07, GusS99, DrmGB15, BerG15, BerG17, AntBG20, GosGB22, DirMH23}
and the references therein.
Many of these works satisfy optimality conditions, provide error bounds, or
offer specialized nonlinear optimization techniques, all with the goal of
providing the most accuracy with a low degree rational function.
Our focus in this paper is on the \AAA{} algorithm and our proposed method,
\AAAA{}, which we demonstrate can provide reliably better accuracy at a lower order
than \AAA{}.

While the \AAA{} method is an adaptive tool that allows the construction of rational approximations to a certain error tolerance, in some cases the error convergence can be unpredictable which might lead to an unnecessarily high degree rational approximant.
Consider as example two functions:
the absolute value $\lvert x \rvert$ and
the rectified linear unit (ReLU) with $\relu(x) = \max(x, 0)$, both in the
interval $[-1, 1]$.
We use $501$ linearly equidistant evaluations of each function to compute
rational approximations.
The results of this experiment are shown in \Cref{fig:IntroExs} in terms of
the normalized $\ltwo$ error vs the degree of the rational functions;
see \Cref{sec:setup} for the definition of the error measure.
While \AAA{} performs very well on the absolute value, it barely converges 
for the ReLU function.
In contrast, we propose in this work the \AAAA{} algorithm, which performs similarly well
compared to \AAA{} for the absolute value and completely outperforms \AAA{}
for ReLU.
Thereby, we showcase that there are more suitable methods for applications in
which low-degree rational functions are needed.
Additionally, we note that in contrast to the classical \AAA{}, the error
behavior of the proposed \AAAA{} method is monotonic.

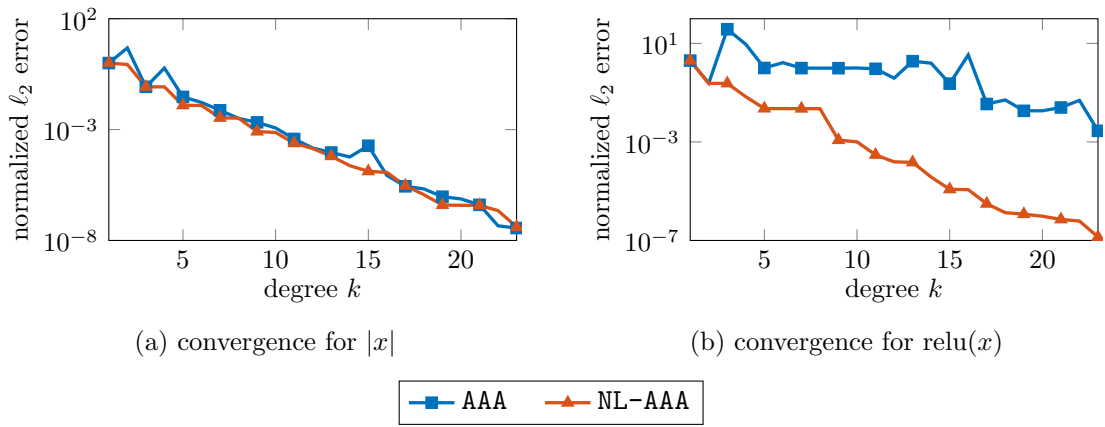
\begin{figure}[t]
  \centering
  \begin{subfigure}[b]{.49\linewidth}
    \centering
  \tikzexternalenable%
  \tikzsetnextfilename{absX}%
  \begin{tikzpicture}[font = \plotfontsize]
  \pgfplotstableread{graphicsPaper/data/introEx_h.csv}\tableERR

  \begin{semilogyaxis}[
    scale only axis,
    width              = .73\linewidth,
    height             = .4\linewidth,
    xmin               = 1,
    xmax               = 23,
    ymin               = 1e-8,
    ymax               = 1e+2,
    xminorticks        = false,
    yminorticks        = true,
    scaled x ticks     = false,
    xlabel             = {degree $k$},
    ylabel             = {normalized $\ltwo$ error},
    xlabel style       = {yshift = .3em},
    ylabel style       = {yshift = -.3em},
    x tick label style = {/pgf/number format/1000 sep={\,}},
    y tick label style = {/pgf/number format/1000 sep={\,}}
  ]
    \addplot[AAAConverge] table[x = iter, y = AAA]{\tableERR};
    \addplot[AAAAConverge] table[x = iter, y = AAAA]{\tableERR};
  \end{semilogyaxis}
\end{tikzpicture}%
  \tikzexternaldisable%

    \caption{convergence for $\lvert x \rvert$}
    \label{fig:absXEx}
  \end{subfigure}%
  \hfill%
  \begin{subfigure}[b]{.49\linewidth}
    \centering
  \tikzexternalenable%
  \tikzsetnextfilename{reluX}%
  \begin{tikzpicture}[font = \plotfontsize]
  \pgfplotstableread{graphicsPaper/data/introEx_g.csv}\tableERR

  \begin{semilogyaxis}[
    scale only axis,
    width              = .73\linewidth,
    height             = .4\linewidth,
    xmin               = 1,
    xmax               = 23,
    ymin               = 1e-7,
    ymax               = 1e+2,
    xminorticks        = false,
    yminorticks        = true,
    scaled x ticks     = false,
    xlabel             = {degree $k$},
    ylabel             = {normalized $\ltwo$ error},
    xlabel style       = {yshift = .3em},
    ylabel style       = {yshift = -.3em},
    x tick label style = {/pgf/number format/1000 sep={\,}},
    y tick label style = {/pgf/number format/1000 sep={\,}}
  ]
    \addplot[AAAConverge] table[x = iter, y = AAA]{\tableERR};
    \addplot[AAAAConverge] table[x = iter, y = AAAA]{\tableERR};
  \end{semilogyaxis}
\end{tikzpicture}%
  \tikzexternaldisable%

    \caption{convergence for $\relu(x)$}
    \label{fig:ReluXEx}
  \end{subfigure}
  
  \vspace{.5\baselineskip}
  \tikzexternalenable%
  \tikzsetnextfilename{legendConvIntro}%
  \begin{tikzpicture}
  \begin{axis}[%
    hide axis,
    scale only axis,
    width  = 1cm,
    height = 1cm,
    xmin   = 0,
    xmax   = 1,
    ymin   = 0,
    ymax   = 1,
    legend columns    = -1,
    legend cell align = {left},
    legend style      = {
      at     = {(0,0)},
      anchor = center,
      /tikz/every even column/.append style = {column sep = 0.4cm}}
  ]
  
    \addlegendimage{AAAConverge} coordinates {(0, 0)};
    \addlegendentry{\AAA}

    \addlegendimage{AAAAConverge}
    \addlegendentry{\AAAA}
  \end{axis}
\end{tikzpicture}%
  \tikzexternaldisable%

  \caption{Convergence behavior of the classical \AAA{} algorithm and the newly
    proposed \AAAA{} method on two test functions:
    In both cases, \AAAA{} decreases the error monotonically.
    While for the absolute value, both methods show a very similar performance,
    for the ReLU function, the proposed method obtains smaller approximation
    errors significantly faster than the classical \AAA{} approach.}
  \label{fig:IntroExs}
\end{figure}

In this work, the task at hand is the construction of rational approximations
$r(z)$ for given data sets of the form $\{(z_{i}, H(z_{i}))\}_{i = 1}^{M}$,
where $z_{i} \in \C$ are the evaluation points and $H(z_{i})$ the corresponding
function evaluations.
Thereby, we aim to reach a prescribed level of accuracy with $r(z)$ with respect
to the given data, while keeping the degree of $r(z)$ as low as possible.
To this end, we propose modifications to the classical \AAA{} algorithm by
solving the true rational least-squares problem in every step using the
Sanathanan-Koerner iteration~\cite{SanK63} and Whitfield's
iteration~\cite{Whi87}.
By deriving the gradients of these different methods in the Wirtinger calculus,
we can show that our proposed \AAAA{} algorithm cheaply solves the
rational least-squares problem of interest in every step rather than the
approximation done by \AAA{}.
Furthermore, our analysis allows us to characterize the performance of the
classical \AAA{} and ensures that our proposed method converges monotonically
with increasing degree of the approximation.

The remainder of this work is organized as follows.
In \Cref{sec:preliminaries}, we review the model order reduction problem and
the fundamentals of \AAA{}, reformulated for the purpose of this paper.
In \Cref{sec:IterativeL2Solve}, we review two iterative methods for solving the
rational least-squares problem, namely the Sanathanan-Koerner iteration and
Whitfield's iteration, and formulate these for the barycentric form.
We then present our proposed algorithm, \AAAA{}, in \Cref{sec:AAAA} and analyze
the different refinement approaches for solving the rational least-squares
problem in terms of their gradients.
Numerical experiments presented in \Cref{sec:numerics} are used to compare the
classical \AAA{} approach and the proposed \AAAA{} method on function
approximation as well as model order reduction examples.
The paper is concluded in \Cref{sec:conclusions}.


\section{Mathematical preliminaries}
\label{sec:preliminaries}

In this section, we give a brief introduction to the model order reduction
problem and its connection to rational approximation.
After introducing the proper interpolatory barycentric form, we review the
classical \AAA{} algorithm for rational approximation from~\cite{NakST18} 
as well as the idea of the Levy approximation~\cite{Lev59} to the rational
nonlinear least-squares problem, which is an essential component in \AAA{}.


\subsection{Model order reduction via rational approximation}%
\label{sec:mor}

Linear time-invariant dynamical systems can be written as systems of
differential and algebraic equations of the form
\begin{subequations} \label{eqn:sys}
\begin{align}
  \bE \dot{\bx}(t) & = \bA \bx(t) + \bb u(t)\\
    y(t) & = \bc^{\trans} \bx(t),
\end{align}
\end{subequations}
where $\bA,\bE \in \C^{N \times N}$ and
$\bb, \bc \in \C^{N}$ are the system matrices,
$\bx\colon \R \to \C^{N}$ is the internal system state,
$u(t)\colon \R \to \C$ is the input to the system, and
$y(t)\colon \R \to \C$ is the external output of the system.
We assume that the matrix pencil $\lambda \bE - \bA$ has at most one infinite
eigenvalue.
Using the Laplace transform, the system~\cref{eqn:sys} can equivalently
be described in the frequency (Laplace) domain via its corresponding transfer
function
\begin{equation} \label{eqn:tf}
  H(z) = \bc^{\trans} (z \bE - \bA)^{-1} \bb.
\end{equation}
The function $H: \C \to \C$ is a degree-$N$ proper rational function, which
yields a direct relation between the system input and output.

The main computational costs for evaluating~\cref{eqn:sys} in applications
stems from the number of differential equations and the corresponding
dimension of the state vector $N$.
In many applications, which demand for high modeling accuracy, the dimension
$N$ is large ($N \in \mathcal{O}(10^{6})$).
Consequently, the evaluation of~\cref{eqn:sys} is associated with high
computational costs in terms of computation time and memory demands.
To mitigate these issues, model order reduction aims to
construct cheap-to-evaluate surrogate models of the same form
\begin{subequations} \label{eqn:rom}
\begin{align}
  \widehat{\bE} \dot{\hat{\bx}}(t) & = \widehat{\bA} \hat{\bx}(t) +
    \hat{\bb} u(t)\\
  \hat{y}(t) & = \hat{\bc}^{\trans} \hat{\bx}(t),
\end{align}
\end{subequations}
where $\widehat{\bE}, \widehat{\bA} \in \C^{k \times k}$,
$\hat{\bb}, \hat{\bc} \in \C^{k}$, and with the internal dimension
$k \ll N$.
The corresponding transfer function is then given as
\begin{equation} \label{eqn:rom_tf}
  r(z) = \hat{\bc}^{\trans} (z \widehat{\bE} - \widehat{\bA})^{-1}
    \hat{\bb}.
\end{equation}
These surrogate models are aimed to accurately approximate the original
system's input-to-output behavior.
In the frequency domain, this approximation problem is equivalent to finding a
rational function of degree $k$ that approximates the transfer function of
the original model~\cref{eqn:tf}; see, for example,~\cite{BenMS05, Ant05}.

Explicit full-order models like~\cref{eqn:sys} are not always readily
accessible; for example, in the cases of complex dynamical processes, which
cannot easily be modeled via first-principle theories.
However, in these cases, typically input-output data is available instead.
In the frequency domain, such data corresponds to evaluations of the transfer
function~\cref{eqn:tf} in given frequency points $\{ z_{i}\}_{i = 1}^{M}$ such that the given data
takes the form of tuples $\{ z_{i}, H(z_{i} \}_{i = 1}^{M}$.
In this setting of data-driven reduced-order modeling, the goal is to
approximate the given data in a suitable way by constructing a
degree-$k$ rational function~\cite{BenSGetal21, AntBG20}.
The approaches described in this work can be applied in both scenarios, i.e.,
for model order reduction and data-driven modeling.


\subsection{The interpolatory barycentric form}%
\label{sec:bary}

A rational function $r\colon \C \to \C$ is said to have the degree $k - 1$
if it can be written as the ratio of degree-$(k-1)$ polynomials.
If the rational function $r$ has no poles in the $k$ complex points
$\lambda_{1}, \ldots, \lambda_{k} \in \C$, then it can be expressed via the
interpolatory barycentric form
\begin{equation} \label{eqn:bary}
  r(z) = \frac{n(z)}{d(z)} = \frac{\sum\limits_{j = 1}^{k}
    w_{j} \frac{h_{j}}{z - \lambda_{j}}}%
    {\sum\limits_{j = 1}^{k} w_{j} \frac{1}{z - \lambda_{j}}}.
\end{equation}
In~\cref{eqn:bary}, the points $\lambda_{1}, \ldots, \lambda_{k} \in \C$ are
called the barycentric support points, 
$h_{1}, \ldots, h_{k} \in \C$ are function values, and
$w_{1}, \ldots, w_{k} \in \C$ are called the barycentric weights.
For any nonzero barycentric weight $w_{j}$, the representation~\cref{eqn:bary}
has a removable singularity in the corresponding expansion point
$\lambda_{j}$ in which it can be expanded to take the corresponding function
value $h_{j}$.
Thus, the continuous extension of the rational function $r$ satisfies the
interpolation conditions $r(z_{j}) = h_{j}$, for $j = 1, \ldots, k$;
see~\cite{BerT04, AntA86}.

For fixed support points and functions values, to employ~\cref{eqn:bary} in the
context of approximating functions or data, the barycentric weights can be seen
as free parameters.
These weights can be used to enforce additional interpolation conditions or to
minimize other error measures with respect to the approximated function or the
given data.
We note that the weights $w_{j}$ in~\cref{eqn:bary} are generally non-unique
since scaling each weight by the same constant does not change the rational
function $r$.
This non-uniqueness is typically resolved by either requiring that the vector of
weights has length $1$ in some suitable norm or by setting one the weights
constant, e.g., $w_{1} = 1$.

In the context of model order reduction, the matrix representation of rational
functions in the form~\cref{eqn:rom_tf} is needed for time domain evaluations
of the corresponding system.
To recover the system matrices from the parameters of the barycentric form,
define
\begin{subequations} \label{eqn:bary_to_rom}
\begin{alignat}{2}
    \widehat{\bE} & = \bmat{1 & -1 & & \\
                \vdots & & \ddots & \\
                1 & & & -1\\
                0 & 0 & \ldots & 0}, & \quad
    \widehat{\bA} & =  \bmat{\lambda_{1} & -\lambda_{2} &  & \\
                \vdots & & \ddots & \\
                \lambda_{1} & & & -\lambda_{k} \\
                -h_{1} w_{1} & -h_{2} w_{2} & \ldots & -h_{k} w_{k}}, \\
    \hat{\bb} & = \bmat{0 & \ldots & 0 & 1}^{\trans}, &
    \hat{\bc} & = \bmat{w_{1} & w_{2} & \ldots & b_{k}}^{\trans}. 
\end{alignat}
\end{subequations}
Then the transfer function $r(z)$ in~\cref{eqn:rom_tf} with the matrices
from~\cref{eqn:bary_to_rom} is the same rational function as the barycentric
form~\cref{eqn:bary} with parameters
$\{ \lambda_{j} \}_{j=1}^k$, $\{ h_{j} \}_{j=1}^k$, and
$\{ w_{j} \}_{j=1}^k$; see, for example,~\cite{Ion13}.

Throughout this work, we will denote the vector of
\emph{numerator basis functions} of the barycentric form as
\begin{equation} \label{eqn:nBasis}
  \nBasis(z) =  \bmat{\frac{h_{1}}{z - \lambda_{1}} & \ldots &
    \frac{h_{k}}{z - \lambda_{k}}}^{\trans} \in \C^{k},
\end{equation}
the vector of \emph{denominator basis functions} as
\begin{equation} \label{eqn:dBasis}
  \dBasis(z) =  \bmat{\frac{1}{z - \lambda_{1}} & \ldots &
    \frac{1}{z - \lambda_{k}}}^{\trans} \in \C^{k},
\end{equation}
and the vector of barycentric weights via
\begin{equation}
  \w = \bmat{w_{1} & \ldots & w_{k}}^{\trans} \in \C^{k}.
\end{equation}
Then, the barycentric form~\cref{eqn:bary} can be compactly written as
\begin{equation} \label{eqn:baryProduct}
  r(z; \w) = \frac{n(z; \w)}{d(z; \w)} =
    \frac{\w^{\trans} \nBasis(z)}{\w^{\trans} \dBasis(z)}.
\end{equation}
Note that the dependence of the rational function $r$ as well as the numerator
and denominator terms on the barycentric weights $\w$ serves the simplicity
of presentation since we will further investigate the effects of different
weight choices.


\subsection{AAA and the Levy approximation}%
\label{sec:aaaLevy}

The \AAA{} algorithm has been established as an efficient and effective approach
for the approximation of given data via rational functions~\cite{NakST18}.
The method is based on a greedy selection of barycentric support
points and the Levy approximation~\cite{Lev59} to the rational least-squares
data fitting problem.

In what follows, we assume that for a general function $H: \C \to \C$,
we have given $M \in \N$ samples of the function evaluation of $H$ in the form
\begin{equation} \label{eqn:givenData}
  \dataSet = \big\{ (z_{1}, H(z_{1})), (z_{2}, H(z_{2})),
    \ldots, (z_{M},H(z_{M})) \big\}.
\end{equation}
The objective of \AAA{} is the iterative construction of rational
approximations to the given data~\cref{eqn:givenData} in the barycentric form
\begin{equation}
  r\AAAiter(z; \w\AAAiter) 
    = \frac{n\AAAiter(z; \w\AAAiter)}{d\AAAiter(z; \w\AAAiter)}
    = \frac{\left( \w\AAAiter \right)^{\trans} \nBasis\AAAiter(z)}%
      {\left( \w\AAAiter \right)^{\trans} \dBasis\AAAiter(z)},
\end{equation}
where $k = 0, 1, \ldots, \maxOrder$ denotes the iteration index.
The iteration may be stopped early when the approximation error between the
current rational approximant $r\AAAiter$ and the data $\dataSet$ is below a
user-defined tolerance $\tau$.
When it is clear from the context, the iteration index $k$ is dropped for
simplicity of notation.

The \AAA{} algorithm is initialized with a rational approximant of degree $0$.
The classical choice for the initialization is the average of the given data, 
\begin{equation}
  r\AAAiter[0](z) \equiv \frac{1}{M} \sum\limits_{i = 1}^{M} H(z_{i}).
\end{equation}
For $k \geq 1$, \AAA{} then determines the support points and weights of a
rational function $r$ in the interpolatory barycentric
form~\cref{eqn:bary}.
At the $k$-th iteration, the location of the maximum mismatch between the
current rational approximant and the data is identified to determine
the $k$-th barycentric support point via
\begin{equation}
  \lambda_{k} = \argmax_{(z_{i}, H(z_{i})) \in \dataSet}
    \lvert r\AAAiter[k-1](z_{i}; \w\AAAiter[k-1]) - H(z_{i}) \rvert.
\end{equation}
With the new support point $\lambda_{k}$ and the corresponding function value
$h_{k} = H(\lambda_{k})$, numerator and denominator bases are updated to
\begin{subequations}
\begin{align}
  \nBasis(z) & = \bmat{\frac{h_{1}}{z - \lambda_{1}} & \ldots &
    \frac{h_{k-1}}{z - \lambda_{k-1}} & \frac{h_{k}}{z - \lambda_{k}}}^{\trans} 
    \quad\text{and} \\
  \dBasis(z) & = \bmat{\frac{1}{z - \lambda_{1}} & \ldots &
    \frac{1}{z - \lambda_{k-1}} & \frac{1}{z - \lambda_{k}}}^{\trans}.
\end{align}
\end{subequations}
Afterwards, \AAA{} aims to use the remaining degrees of freedom via the
weights in $\w$ to minimize the least-squares error of the approximation
\begin{equation} \label{eqn:L2Err_AAA}
  E = \sum\limits_{i = 1}^{M}
    \left\lvert r(z_{i}; \w) - H(z_{i}) \right\rvert^{2}.
\end{equation}
Since~\cref{eqn:L2Err_AAA} is nonlinear in the unknown weights $\w$,
instead of solving~\cref{eqn:L2Err_AAA}, \AAA{} minimizes the corresponding
Levy approximation of the error, namely 
\begin{equation} \label{eqn:LevyErrCrit}
  \ELevy = \sum\limits_{i = 1}^{M}
    \left\lvert n(z_{i}; \w) - d(z_{i}; \w) H(z_{i}) \right\rvert^{2},
\end{equation}
where $n$ and $d$ are the numerator and denominator of the barycentric form
of $r$.
Hence, the weights $\w$ are found via
\begin{subequations} \label{eqn:LevyL2Prob_noMatrix}
\begin{align}
  \w & = \argmin\limits_{\bv \in \C^{k}} \sum\limits_{i = 1}^{M}
    \lvert n(z_{i}; \bv) - d(z_{i}; \bv) H(z_{i}) \rvert^{2} \\
  & = \argmin\limits_{\bv \in \C^{k}} \sum_{i = 1}^{M}
    \lvert \bv^{\trans} \nBasis(z_{i}) - \bv^{\trans} \dBasis(z_{i}) H(z_{i})
    \rvert^{2}.
\end{align}
\end{subequations}
The expression in~\cref{eqn:LevyL2Prob_noMatrix} is a homogeneous
linear least-squares problem, and thus has the trivial solution $\w = \zeros$.
This is an artifact of the non-uniqueness of the weights in the barycentric
form.
We follow the original proposition in~\cite{NakST18} and circumvent this
issue by enforcing that $\lVert \w \rVert_{2} = 1$.

To effectively solve~\cref{eqn:LevyL2Prob_noMatrix} with the constraint
$\lVert \w \rVert_{2} = 1$, we first note that as a result of the interpolation
property of the barycentric form, the approximation errors of $r$ in the
support points $\lambda_{1}, \ldots, \lambda_{k}$ is zero.
Therefore, we will exclude the $k$ interpolated data samples from the data
set $\dataSet$ for the solution of the linear least-squares problem.
For simplicity of presentation, we re-index the remaining (non-interpolated)
data samples in $\dataSet$ to be $\{ (z_{i}, H(z_{i})) \}_{i = 1}^{M - k}$.
Now, define the Cauchy matrix $\CauchyMat \in \C^{(M - k) \times k}$ via
\begin{equation} \label{eqn:CauchyMatDef}
  \CauchyMat_{i, j} = \frac{1}{z_{i} - \lambda_{j}},
    \quad\text{for}\quad
    i = 1, \ldots, M - k
    \quad\text{and}\quad
    j = 1, \ldots, k,
\end{equation}
define the diagonal matrix of interpolated function values as
\begin{equation} \label{eqn:interpFunValsMatDef}
  \interpFunVals = \mdiag(h_{1}, h_{2}, \ldots, h_{k} ) \in \C^{k \times k},
\end{equation}
and define the diagonal matrix of non-interpolated data values by
\begin{equation} \label{eqn:FunValsMatDef}
  \FunValsMat = \mdiag(H(z_{1}), H(z_{2}), \ldots, H(z_{M - k})) \in
    \C^{(M - k) \times (M - k)}.
\end{equation}
Then, the solution to~\cref{eqn:LevyL2Prob_noMatrix}, with
$\lVert \w \rVert_{2} = 1$, can equivalently be obtained from
\begin{equation} \label{eqn:LevyL2Prob_Mat}
  \w = \argmin\limits_{\bv \in \C^{k}, \lVert \bv \rVert_{2} = 1}
    \left\lVert \left( \FunValsMat \CauchyMat - \CauchyMat \interpFunVals
    \right) \bv \right\rVert_{2}^{2}.
\end{equation}
Finally, the solution to~\cref{eqn:LevyL2Prob_Mat} is simply given as the
$k$-th right-singular vector of the
matrix~$\FunValsMat \CauchyMat - \CauchyMat \interpFunVals$.
The complete \AAA{} algorithm is summarized in \Cref{alg:AAA}.

\begin{algorithm}[t]
  \SetAlgoHangIndent{1pt}
  \DontPrintSemicolon
  
  \caption{Adaptive Antoulas-Anderson (\AAA) algorithm.}%
  \label{alg:AAA}

  \KwIn{Data set $\dataSet = \{(z_{i}, H(z_{i}))\}_{i = 1}^{M}$,
    error tolerance $\tau$,
    maximum degree of rational approximant $\maxOrder$.}
  \KwOut{Barycentric parameters
    $\{ h_{j} \}_{j = 1}^{k}$,
    $\{ \lambda_{j} \}_{j = 1}^{k}$,
    $\{w_{j}\}_{j = 1}^{k}$.}

  Initialize
    $r\AAAiter[0](z_{i}) \equiv \frac{1}{M} \sum_{i = 1}^{M} H(z_{i})$,
    $\nBasis\AAAiter[0] = [~]$, and
    $\dBasis\AAAiter[0] = [~]$.\;
  
  \For{$k = 1, \ldots, \maxOrder + 1$}{
    Determine the next support point and function value
      \vspace{-.5\baselineskip}
      \begin{equation*}
        (\lambda_{k}, h_{k}) = \argmax\limits_{(z_{i}, H(z_{i})) \in \dataSet}
          \lvert r\AAAiter[k - 1](z_{i}; \w\AAAiter[k - 1]) - H(z_{i}) \lvert.
      \end{equation*}\;
      \vspace{-1.5\baselineskip}

    Update the basis vectors with $(\lambda_{k}, h_{k})$ so that
      \vspace{-.5\baselineskip}
      \begin{align*}
        \nBasis\AAAiter(z) & =
          \bmat{\frac{h_{1}}{z - \lambda_{1}} & \ldots &
          \frac{h_{k-1}}{z - \lambda_{k-1}} &
          \frac{h_{k}}{z - \lambda_{k}}}^{\trans},\\
        \dBasis\AAAiter(z) & =
          \bmat{\frac{1}{z-\lambda_1} & \ldots &
          \frac{1}{z-\lambda_{k-1}} & \frac{1}{z-\lambda_k}}^{\trans}.
      \end{align*}\;
      \vspace{-1.5\baselineskip}

    Update the data set
      $\dataSet \gets \dataSet \setminus \{ (\lambda_{k}, h_{k}) \}$.\;

    Form the matrices $\CauchyMat, \interpFunVals, \FunValsMat$
      via~\cref{eqn:CauchyMatDef,eqn:interpFunValsMatDef,eqn:FunValsMatDef}.\; 

    Solve the constrained linear least-squares problem
      \vspace{-.5\baselineskip}
      \begin{equation*}
        \w\AAAiter = \argmin\limits_{\bv \in \C^{k}, \lVert \bv \rVert_{2} = 1}
          \left\lVert \left( \FunValsMat \CauchyMat -
          \CauchyMat \interpFunVals \right) \bv \right\rVert_{2}^{2}.
      \end{equation*}\;
      \vspace{-1.5\baselineskip}

    \lIf{$\sum_{i = 1}^{M - k}
      \left\lvert r\AAAiter(z_{i}; \w\AAAiter) -
      H(z_{i}) \right\rvert^2 < \tau$}{\textbf{break}}
  }
\end{algorithm}


\section{Iterative refinements to the Levy approximation}%
\label{sec:IterativeL2Solve}

The Levy approximation~\cref{eqn:LevyErrCrit} to the rational least-squares
problem~\cref{eqn:L2Err_AAA} enables the rapid construction of rational
approximants in the \AAA{} algorithm.
However, the solution to the linearized least-squares problem may not 
coincide well with the solution to the true rational least-squares problem.
To resolve this, we present in this section two iterative refinement approaches
for Levy's approximation, which improve the approximation accuracy by solving a
sequence of linear least-squares problems.


\subsection{Sanathanan-Koerner iteration}%
\label{sec:SK}

In general, for rational functions of the form~\cref{eqn:baryProduct}, Levy's
approximation to the rational least-squares problem~\cref{eqn:L2Err_AAA} can
be derived using the reformulation of the individual error
terms in~\cref{eqn:L2Err_AAA} by factoring out the denominators such as
\begin{equation}
  \left\lvert \frac{n(z; \w)}{d(z; \w)} - H(z) \right\rvert = 
    \left\lvert \frac{1}{d(z; \w)} \right\rvert
    \left\lvert n(z; \w) - d(z; \w) H(z) \right\rvert.
\end{equation}
Then, Levy's approximation (to the error) is given by
\begin{equation} \label{eqn:LevyScale}
  \left\lvert \frac{1}{d(z; \w)} \right\rvert
    \left\lvert n(z; \w) - d(z; \w) H(z) \right\rvert
    \approx \left\lvert n(z; \w) - d(z; \w) H(z) \right\rvert,
\end{equation}
that is, the nonlinear objective function~\cref{eqn:L2Err_AAA} is scaled by the
reciprocals of the absolute values of the denominator.

In~\cite{SanK63}, Sanathanan and Koerner developed an iteration scheme
(further on referred to as Sanathanan-Koerner iteration or simply
\SK{} iteration) that aims to undo the scaling introduced by the Levy
approximation~\cref{eqn:LevyScale}.
In the $p$-th step of the iteration, for a degree-$(k - 1)$ rational function
in barycentric form, \SK{} finds new weights $\w\skwfIter \in \C^{k}$ as
minimizer of the weighted linear least-squares error
\begin{equation} \label{eqn:SKApx}
  \Esk = \sum\limits_{i = 1}^{M} \frac{1}{\lvert d(z_{i}; \w\skwfIter[p - 1])
    \rvert^{2}} \lvert n(z_{i}; \w\skwfIter) - d(z_{i}; \w\skwfIter)
    H(z_{i}) \rvert^{2},
\end{equation}
where the weighting is given by the $d(z_{i}; \w\skwfIter[p - 1])$ using the
weights from the previous iteration step.
The iteration is initialized with the constant denominator
$d(z; \w\skwfIter[0]) \equiv 1$ such that in the first step of \SK{} the
solution to the Levy approximation~\cref{eqn:LevyErrCrit} is computed.

The original work~\cite{SanK63} describes the numerical procedure with a
generic polynomial basis.
Here, we present the necessary computations in the interpolatory barycentric
form instead.
For fixed numerator basis $\nBasis$ and denominator basis $\dBasis$, at every
iteration step $p \geq 1$, the updated weights $\w\skwfIter$ are given by
\begin{equation} \label{eqn:SKMinProb}
  \w\skwfIter = \argmin\limits_{\bv \in \C^{k}, \lVert \bv \rVert_{2} = 1}
    \sum\limits_{i = 1}^{M} \frac{1}{\lvert \w\skwfIter[p - 1]^{\trans}
    \dBasis(z_{i}) \rvert^{2}}
    \lvert \bv^{\trans} \nBasis(z_{i}) - \bv^{\trans} \dBasis(z_{i}) H(z_{i})
    \rvert^{2},
\end{equation}
where the non-uniqueness of the weights in the barycentric form is resolved via
the added constraint $\lVert \w\skwfIter \rVert_{2} = 1$.

Similar to the Levy approximated least-squares problem, we may
solve~\cref{eqn:SKMinProb} via a reformulation into matrix form.
To this end, we construct the same matrices
$\CauchyMat \in \C^{(M - k) \times k}$,
$\interpFunVals \in \C^{(M - k) \times (M - k)}$, and 
$\FunValsMat \in \C^{k \times k}$ as
in~\cref{eqn:CauchyMatDef,eqn:interpFunValsMatDef,eqn:FunValsMatDef}, and
we define the new diagonal weighting matrix
\begin{equation} \label{eqn:Ddef_SK}
  \SKWFWeightMat = \mdiag\left(
    \frac{1}{\lvert d(z_{1}; \w\skwfIter[p - 1]) \rvert}, \ldots,
    \frac{1}{\lvert d(z_{M - k}; \w\skwfIter[p - 1]) \rvert}
    \right) \in \R^{(M - k)\times (M - k)}.
\end{equation}
Then, the solution to~\cref{eqn:SKMinProb} is equivalent to the solution
to
\begin{equation} \label{eqn:SKMinProb_MatForm}
  \w\skwfIter = \argmin\limits_{\bv \in \C^{k}, \lVert \bv \rVert_{2} = 1}
    \left\lVert \SKWFWeightMat \LowMat \bv \right\rVert_{2}^{2},
\end{equation}
which is given by the $k$-th right-singular vector of the matrix
$\SKWFWeightMat \LowMat$.

The \SK{} iteration is considered to have converged when the change in the
weight vectors of consecutive iteration steps is small enough.
In general, there are no convergence guarantees for the \SK{} iteration, and
the approach has been observed to not necessarily converge for all possible
choices of barycentric support points and function values.
Therefore, in our implementation of the \SK{} iteration, we additionally
evaluate the rational least-squares error~\cref{eqn:L2Err_AAA} at every 
iteration step and return weights $\w\skwfIter$ corresponding to the lowest
rational least-squares error of all computed iterates.
The \SK{} iteration is summarized in \Cref{alg:SK}.

\begin{remark}
  The coefficient matrix in~\cref{eqn:SKMinProb_MatForm} is known to become
  highly ill-conditioned.
  A variety of approaches has been explored to resolve this conditioning
  issue in the standard \SK{} iteration leading, for example, to the vector
  fitting method~\cite{GusS99}, the stabilized \SK{} iteration~\cite{Hok20},
  or the orthogonal rational approximation~\cite{MaE22}.
  However, all of these methods rely on the ability to change the basis
  vectors~\cref{eqn:nBasis,eqn:dBasis} of the rational function at each
  iteration.
  As we require a specific barycentric basis chosen to enforce interpolation
  conditions, these techniques are not applicable to this work.
\end{remark}

\begin{algorithm}[t]
  \SetAlgoHangIndent{1pt}
  \DontPrintSemicolon
  
  \caption{Sanathanan-Koerner (\SK) iteration.}%
  \label{alg:SK}

  \KwIn{Data set $\dataSet = \{(z_{i}, H(z_{i}))\}_{i = 1}^{M - k}$,
    convergence tolerance $\tau_{\SK}$,
    barycentric parameters $\{ \lambda_{j} \}_{j = 1}^{k}$ and
      $\{ h_{j} \}_{j = 1}^{k}$,
    maximum number of iterations $\maxOrderSK$.}
  \KwOut{Barycentric weights $\w$.}

  Initialize $d\skwfIter[0](z) \equiv 1$,
    $\be_{\ltwo} = [~]$, and
    $\w\skwfIter[0] = \zeros$.\;

  Form the matrices $\CauchyMat, \interpFunVals, \FunValsMat$
    via~\cref{eqn:CauchyMatDef,eqn:interpFunValsMatDef,eqn:FunValsMatDef}.\;
  
  \For{$p = 1, \ldots, \maxOrderSK$}{
    Compute the weighting matrix $\SKWFWeightMat$ using~\cref{eqn:Ddef_SK}.\;

    Solve the constrained linear least-squares problem
      \vspace{-.5\baselineskip}
      \begin{equation*}
        \w\skwfIter = \argmin\limits_{\bv \in \C^{k}, \lVert \bv \rVert_{2} = 1}
          \left\lVert \SKWFWeightMat \LowMat \bv \right\rVert_{2}^{2}.
      \end{equation*}\;
      \vspace{-1.5\baselineskip}

    Expand the error vector
      $\be_{\ltwo} \gets \bmat{\be_{\ltwo} &
      \sum_{i = 1}^{M - k}
      \left\lvert r(z_{i}; \w\skwfIter) - H(z_{i}) \right\rvert^{2}}$.\;

    \lIf{$\lVert \w\skwfIter - \w\skwfIter[p - 1] \rVert_{2} < \tau_{\SK}$}%
      {\textbf{break}}
  }

  Find the index $p_{\operatorname{best}}$ of the smallest entry in
    $\be_{\ltwo}$.\;

  Set $\w \gets \w\skwfIter[p_{\operatorname{best}}]$.\;
\end{algorithm}


\subsection{Whitfield's iteration}%
\label{sec:wf}

Another refinement procedure for the solution of the rational least-squares
problem was proposed by Whitfield in~\cite{Whi87}.
Here, we re-derive the approach using the barycentric
representation~\cref{eqn:bary} while the original work used a generic
polynomial basis.
Similar to the \SK{} iteration, Whitfield's (\WF{}) iteration seeks to minimize
the rational least-squares error~\cref{eqn:L2Err_AAA} by solving a sequence of
linear least-squares approximations starting from an initial weight vector
$\w\skwfIter[0]$.
In the following, we assume that
$d(z_{i}; \w) = \w^{\trans} \dBasis(z_{i}) \neq 0$ holds for all $z_{i}$ in
the sampled data set $\dataSet$ from~\cref{eqn:givenData}.

In the $p$-th step of the \WF{} iteration, we form a linear approximation of
the rational function $r$ along the weights centered at the previous weights
$\w\skwfIter[p - 1]$.
This particular linearization of the rational function $r$ is given by
\begin{equation} \label{eqn:LinApxGeneral}
  r(z; \w) \approx r(z; \w\skwfIter[p - 1]) +
    \left(\left. \drdw \right|_{\w = \w\skwfIter[p - 1]} \right)^{\trans}
    (\w - \w\skwfIter[p - 1]).
\end{equation}
Therein, the partial derivative
$r(z; \w) = (\w^{\trans} \nBasis(z))/(\w^{\trans} \dBasis(z))$ with respect
to the weights~$\w$ takes the form
\begin{equation} \label{eqn:totalDerivative}
  \drdw = \frac{1}{d(z; \w)} \big( \nBasis(z) - r(z; \w) \dBasis(z) \big).
\end{equation}
Inserting~\cref{eqn:totalDerivative} into~\cref{eqn:LinApxGeneral} then yields
the linear approximation of the form
\begin{subequations}
\begin{align}
  r(z; \w) & \approx r(z; \w\skwfIter[p - 1]) +
    \left( \frac{1}{d(z; \w\skwfIter[p - 1])}
    \left( \nBasis(z) - r(z; \w\skwfIter[p - 1]) \dBasis(z) \right)
    \right)^{\trans}(\w - \w\skwfIter[p - 1]) \\
  \label{eqn:wf_linearization}
  & = \frac{1}{d(z; \w\skwfIter[p - 1])}
    \Big( n(z; \w) - r(z; \w\skwfIter[p - 1]) d(z; \w) +
    n(z; \w\skwfIter[p - 1]) \Big).
\end{align}
\end{subequations}
Finally, substituting the expression~\cref{eqn:wf_linearization} in the
rational least-squares error~\cref{eqn:L2Err_AAA} gives the new
linear least-squares error
\begin{equation} \label{eqn:WFError}
  \begin{aligned}
    E_{\WF} = \sum\limits_{i = 1}^{M} &
      \frac{1}{\lvert d(z_{i}; \w\skwfIter[p - 1]) \rvert^{2}}
      \Big\lvert n(z_{i}; \w) - r(z_{i}; \w\skwfIter[p - 1]) d(z_{i}; \w) \\
    & {}+{}
      n(z_{i}; \w\skwfIter[p - 1]) - d(z_{i}; \w\skwfIter[p - 1]) H(z_{i})
      \Big\rvert^{2}.
  \end{aligned}
\end{equation}
Then, for every step $p \geq 1$ of the iteration, the new weights
$\w\skwfIter$ are calculated as the solution to
\begin{equation} \label{eqn:WFMinProb}
  \begin{aligned}
    \w\skwfIter = \argmin\limits_{\bv \in \C^{k}}
      \sum\limits_{i = 1}^{M} 
      \frac{1}{\lvert \w\skwfIter[p - 1]^{\trans} \dBasis(z_{i}) \rvert^{2}} &
      \left\lvert \bv^{\trans} \nBasis(z_{i}) -
      \frac{\w\skwfIter[p - 1]^{\trans} \nBasis(z_{i})}%
        {\w\skwfIter[p - 1]^{\trans} \dBasis(z_{i})}
      \bv^{\trans} \dBasis(z_{i}) \right.\\
  & ~{}+{}
    \left.\vphantom{\frac{\w\skwfIter[p - 1]^{\trans}}%
      {\w\skwfIter[p - 1]^{\trans}}}
    \w^{\trans}\skwfIter[p - 1] \nBasis(z_{i}) -
    \w\skwfIter[p - 1]^{\trans} \dBasis(z_{i}) H(z_{i}) \right\rvert^{2}.
  \end{aligned}
\end{equation}

We note that the least-squares error in~\cref{eqn:WFError} depends on both the
previous numerator $n(z; \w\skwfIter[p - 1])$ and the previous denominator
$d(z; \w\skwfIter[p - 1])$.
Initializing both of these quantities for the iteration is a non-trivial task.
In particular, initializing both numerator and denominator identical to $1$
does not result in the Levy approximation~\cref{eqn:LevyErrCrit}.
We explore initialization strategies for the \WF{} iteration later in
\Cref{sec:WFInits}.

While both the Levy approximation and \SK{} iteration require the solution of
constrained homogeneous linear least-squares
problems~\cref{eqn:LevyL2Prob_Mat,eqn:SKMinProb_MatForm}, the \WF{} iteration
calls for the solution of non-homogeneous linear least-squares problems of the
form~\cref{eqn:WFMinProb} and hence does not yield a trivial solution.
However, \cref{eqn:WFMinProb}  does not have a unique solution either due to the
non-uniqueness of the weights in the barycentric form~\cref{eqn:bary}.
To resolve the non-uniqueness of the least-squares solution in the \WF{}
iteration, we enforce the first weight vector entry to be $1$, i.e., we add
the constraint
\begin{equation} \label{eqn:WF_wConstraint}
  \left[ \w\skwfIter \right]_{1} = 1.
\end{equation}
For the solution of~\cref{eqn:WFMinProb} with the
constraint~\cref{eqn:WF_wConstraint}, we assemble the weight matrix
$\SKWFWeightMat$ as in~\cref{eqn:Ddef_SK}, and define the new matrix
$\WFCoefMat \in \C^{(M - k)\times k}$ entrywise via
\begin{equation} \label{eqn:WFCoefMatDef}
  \left[ \WFCoefMat \right]_{i, j} = \frac{h_{j}}{z_{i} - \lambda_{j}} -
    r(z_{i}; \w\skwfIter[p - 1]) \frac{1}{z_{i} - \lambda_{j}},
\end{equation}
for $i = 1, \ldots M - k$ and $j = 1, \ldots, k$,
and we define the vector $\WFRHS \in \C^{(M - k)}$ as
\begin{equation} \label{eqn:WFRHSDef}
  \left[ \WFRHS \right]_{i} = -n(z_{i}; \w\skwfIter[p - 1]) +
    d(z_{i}; \w\skwfIter[p - 1]) H(z_{i}).
\end{equation}
for $i = 1, \ldots M - k$.
Similar to the construction of $\SKWFWeightMat$ in~\cref{eqn:Ddef_SK}, the data
samples used in~\cref{eqn:WFCoefMatDef,eqn:WFRHSDef} do not include the
interpolated data samples that are used as parameters in the barycentric form.
Using these matrices, the solution to~\cref{eqn:WFMinProb} with the
constraint~\cref{eqn:WF_wConstraint} is given by
\begin{equation} \label{eqn:WFminProb_matrix}
  \w\skwfIter = \argmin\limits_{\bv \in \C^{k}, \bv_{1} = 1}
    \left\lVert \SKWFWeightMat \left( \left[ \WFCoefMat \right]_{:, 2:k}
    \left[ \bv \right]_{2:k} -
    \big( \WFRHS - \left[ \WFCoefMat \right]_{:, 1} \big) \right)
    \right\rVert_{2}^{2},
\end{equation}
where $\left[ \WFCoefMat \right]_{:, 2:k}$ denotes the last $k - 1$ columns
of the matrix $\WFCoefMat$, while $\left[ \WFCoefMat \right]_{:, 1}$ denotes
the first column of $\WFCoefMat$, and $\left[ \bv \right]_{2:k}$ marks the
last $k - 1$ rows of the vector $\bv$.
\Cref{eqn:WFminProb_matrix} can be seen as a classical unconstrained
non-homogeneous weighted linear least-squares problem in $k - 1$ variables,
and it can be solved by any off-the-shelf solver.

Similar to the case of the \SK{} iteration, there are no convergence guarantees
for the \WF{} iteration.
As such, we employ the same strategy in the implementation of the \WF{}
iteration to save the rational least-squares errors~\cref{eqn:L2Err_AAA}
at every iteration step and then return the weights corresponding to the
smallest error.
The \WF{} iteration is outlined in \Cref{alg:WF}.

\begin{algorithm}[t]
  \SetAlgoHangIndent{1pt}
  \DontPrintSemicolon
  
  \caption{Whitfield's (\WF) iteration.}%
  \label{alg:WF}

  \KwIn{Data set $\dataSet = \{(z_{i}, H(z_{i}))\}_{i = 1}^{M - k}$,
    convergence tolerance $\tau_{\WF}$,
    barycentric parameters $\{ \lambda_{j} \}_{j = 1}^{k}$ and
      $\{ h_{j} \}_{j = 1}^{k}$,
    initial weights $\w\skwfIter[0]$,
    maximum number of iterations $\maxOrderSK$.}
  \KwOut{Barycentric weights $\w$.}

  Initialize $\be_{\ltwo} = [~]$.\;
  
  \For{$p = 1, \ldots, \maxOrderSK$}{
    Construct $\SKWFWeightMat, \WFCoefMat$ and $\WFRHS$
      using~\cref{eqn:Ddef_SK,eqn:WFCoefMatDef,eqn:WFRHSDef}.\;
  
    Solve the linear least-squares problem
    \vspace{-.5\baselineskip}
    \begin{equation*}
      \w\skwfIter = \argmin\limits_{\bv \in \C^{k}, \bv_{1} = 1}
        \left\lVert \SKWFWeightMat \left( \left[ \WFCoefMat \right]_{:, 2:k}
        \left[ \bv \right]_{2:k} -
        \big( \WFRHS - \left[ \WFCoefMat \right]_{:, 1} \big) \right)
        \right\rVert_{2}^{2}.
    \end{equation*}\;
    \vspace{-1.5\baselineskip}

    Expand the error vector
      $\be_{\ltwo} \gets \bmat{\be_{\ltwo} &
      \sum_{i = 1}^{M - k}
      \left\lvert r(z_{i}; \w\skwfIter) - H(z_{i}) \right\rvert^{2}}$.\;

    \lIf{$\lVert \w\skwfIter - \w\skwfIter[p - 1] \rVert_{2} < \tau_{\WF}$}%
      {\textbf{break}}
  }
  
  Find the index $p_{\operatorname{best}}$ of the smallest entry in
    $\be_{\ltwo}$.\;

  Set $\w \gets \w\skwfIter[p_{\operatorname{best}}]$.\;
\end{algorithm}


\section{Nonlinear least-squares refinement for AAA}%
\label{sec:AAAA}

In principle, the two steps performed in the classical \AAA{} are the greedy
selection of the next interpolation point from the data and the subsequent
fit of the barycentric weights using the Levy approximation of the rational
least-squares problem.
We propose the replacement of the Levy approximation in \AAA{} by the iterative
refinement methods discussed earlier.
In this section, we first describe our proposed new \AAA{}-type algorithm
before we analyze the potential effectiveness of the different refinement
methods integrated into our proposed algorithm for the solution of the true
nonlinear least-squares problem.


\subsection{Algorithm description}%
\label{sec:method}

Using the machinery outlined in \Cref{sec:IterativeL2Solve}, we present our
Nonlinear Least-squares \AAA{} algorithm (\AAAA{}) in \Cref{alg:AAAA}.
While the first steps of \Cref{alg:AAA} (\AAA{}) and \Cref{alg:AAAA} are identical,
the solution to the constrained linear least-squares problem in \Cref{alg:AAA}
is replaced by
\hyperref[alg:AAAA_SK]{Lines~\ref{alg:AAAA_SK}}--\ref{alg:AAAA_endif}
in \Cref{alg:AAAA}.
First, a refined set of weights $\w_{\SK}$ is computed using the
\SK{} iteration from \Cref{alg:SK} in \Cref{alg:AAAA_SK}.
Then, a single iteration step of the \WF{} iteration from \Cref{alg:WF} is
performed in \Cref{alg:AAAA_WF1} using the previous set of weights extended by
zero, $\sbmat{\left(\w\AAAiter[k-1]\right)^{\trans} & 0}^{\trans}$, as initialization to
compute an alternative set of weights~$\w\skwfIter[1]$.
Afterwards, the two weight vectors are compared in terms of the corresponding
approximation errors to determine how to initialize a complete run of the
\WF{} iteration.
If $\w_{\SK}$ gives the smaller error, then \Cref{alg:WF} is run with the
initialization $\w\skwfIter[0] = \w_{\SK}$ to compute $\w\AAAiter$.
Otherwise, \Cref{alg:WF} is run with the initialization
$\sbmat{\left(\w\AAAiter[k-1]\right)^{\trans} & 0}^{\trans}$.
The motivation for this specific initialization strategy and alternative
strategies in the case that the iterative refinement strategies fail are
explained in detail in \Cref{sec:WFInits}.

\begin{algorithm}[t]
  \SetAlgoHangIndent{1pt}
  \DontPrintSemicolon
  
  \caption{Nonlinear Least-squares \AAA{} (\AAAA) algorithm.}%
  \label{alg:AAAA}

  \KwIn{Data set $\dataSet = \{(z_{i}, H(z_{i}))\}_{i = 1}^{M}$,
    error tolerance $\tau$,
    maximum number of refinement iterations $\maxOrderSK$,
    refinement convergence tolerances $\tau_{\SK}$ and $\tau_{\WF}$,
    maximum degree of rational approximant~$\maxOrder$.}
  \KwOut{Barycentric parameters
    $\{ h_{j} \}_{j = 1}^{k}$,
    $\{ \lambda_{j} \}_{j = 1}^{k}$,
    $\{w_{j}\}_{j = 1}^{k}$.}

  Initialize
    $r\AAAiter[0](z_{i}) \equiv \frac{1}{M} \sum_{i = 1}^{M} H(z_{i})$,
    $\nBasis\AAAiter[0] = [~]$, and
    $\dBasis\AAAiter[0] = [~]$.\;
  
  \For{$k = 1, \ldots, \maxOrder + 1$}{
    Determine the next support point and function value
      \vspace{-.5\baselineskip}
      \begin{equation*}
        (\lambda_{k}, h_{k}) = \argmax\limits_{(z_{i}, H(z_{i})) \in \dataSet}
          \lvert r\AAAiter[k - 1](z_{i}; \w\AAAiter[k - 1]) - H(z_{i}) \lvert.
      \end{equation*}\;
      \vspace{-1.5\baselineskip}

    Update the basis vectors with $(\lambda_{k}, h_{k})$ so that
      \vspace{-.5\baselineskip}
      \begin{align*}
        \nBasis\AAAiter(z) & =
          \bmat{\frac{h_{1}}{z - \lambda_{1}} & \ldots &
          \frac{h_{k-1}}{z - \lambda_{k-1}} &
          \frac{h_{k}}{z - \lambda_{k}}}^{\trans},\\
        \dBasis\AAAiter(z) & =
          \bmat{\frac{1}{z-\lambda_1} & \ldots &
          \frac{1}{z-\lambda_{k-1}} & \frac{1}{z-\lambda_k}}^{\trans}.
      \end{align*}\;
      \vspace{-1.5\baselineskip}

    Update the data set
      $\dataSet \gets \dataSet \setminus \{ (\lambda_{k}, h_{k}) \}$.\;
      \label{alg:AAAA_max}

    Compute $\w_{\SK}$ via $\maxOrderSK$ iterations of
      \Cref{alg:SK} and tolerance $\tau_{\SK}$.\;
      \label{alg:AAAA_SK}

    Compute $\w\skwfIter[1]$ via $1$ iteration of \Cref{alg:WF}
      using $\bmat{\left(\w\AAAiter[k-1]\right)^{\trans} & 0}^{\trans}$.\;
      \label{alg:AAAA_WF1}

    \eIf{$\sum_{i = 1}^{M - k}
      \left\lvert r\AAAiter(z_{i}; \w_{\SK}) - H(z_{i}) \right\rvert^2
      <
      \sum_{i = 1}^{M - k}
      \left\lvert r\AAAiter(z_{i}; \w\skwfIter[1]) - H(z_{i}) \right\rvert^2$}{
        Compute the weights $\w\AAAiter$ via $\maxOrderSK$ iterations
          of \Cref{alg:WF} using $\w\skwfIter[0] = \w_{\SK}$ and
          tolerance $\tau_{\WF}$.\;
        \label{alg:AAAA_init1}
    }{
        Compute the weights $\w\AAAiter$ via $\maxOrderSK$ iterations of
          \Cref{alg:WF} using $\w\skwfIter[0] =
          \bmat{\left(\w\AAAiter[k - 1]\right)^{\trans} & 0}^{\trans}$
          and tolerance $\tau_{\WF}$.\;
        \label{alg:AAAA_init2}
    }
    \label{alg:AAAA_endif}

    \lIf{$\sum_{i = 1}^{M - k}
      \left\lvert r\AAAiter(z_{i}; \w\AAAiter) -
      H(z_{i}) \right\rvert^{2} < \tau$}{\textbf{break}}
  }
\end{algorithm}

In terms of additional computational costs for \Cref{alg:AAAA} compared
to \Cref{alg:AAA}, we note that the main costs in both algorithms stem from the
solution of linear least-squares problems.
Thereby, in the case that neither refinement procedure converges early,
\Cref{alg:AAAA} is solving two sequences of linear least-squares
problems of the same dimensions as the single one in \Cref{alg:AAA} making
\Cref{alg:AAAA} computationally more expensive than the classical method in
terms of computation time.
However, the amount of additional computation time is bounded by the
maximum number of iterations in the refinement procedures so that each step of
\Cref{alg:AAAA} is at most $2\maxOrderSK$ times as expensive as the
corresponding step in \Cref{alg:AAA}.
Therefore, our proposed method lies in the same order of computational
complexity as the original method.
As we will show in various numerical examples, \AAAA{} typically achieves a
given tolerance at a lower degree than AAA; thus even though the individual
iteration step is computationally more expensive, in many cases \AAAA{}
achieves the given target tolerance with a smaller number of iteration steps.


\subsection{Analysis of the approximation and refinement methods}%
\label{sec:optimal}

In this section, we provide the theoretical justification for the design of our
proposed \AAAA{} algorithm (\Cref{alg:AAAA}).
After a quick introduction of the Wirtinger calculus, we derive and analyze
the gradients of the approximations done by Levy, \SK, and \WF.
While all these approaches strive to reduce the rational least-squares
error~\cref{eqn:L2Err_AAA}, we show that only \WF{} minimizes this particular
error.
In contrast to the work in~\cite{Whi87}, we rigorously treat the nonanalytic
error criteria using the barycentric representation~\cref{eqn:bary} and also
provide additional analysis.


\subsubsection{Wirtinger calculus}%
\label{sec:wirtinger}

Our main tool for the derivation of the gradients of the error
function~\cref{eqn:L2Err_AAA} is the Wirtinger calculus~\cite{Wir27}.
As the objective function and approximations to it are in general mappings
from $\C^{M}$ to $\R$, they are not complex analytic and hence cannot be
differentiated in the notion of classical complex derivatives.
However, these functions are differentiable using the Wirtinger calculus, which
was specifically designed as a calculus for functions which are not
complex-differentiable when regarded as functions from $\C^{M}$ to $\C$
(resp. $\R$), but are real-differentiable when regarded as functions from
$\R^{2M}$ to $\R^{2}$ (resp. $\R$).
To this end, the Wirtinger calculus re-interprets a function
$f:\C^{M} \to \R$ that depends on the complex variable $\bz \in \C^{M}$ as
a function of $\bz$ and $\overline{\bz}$, where $\bz$ and $\overline{\bz}$
are treated as independent variables.
Equivalently, the function $f$ can be considered as a function of the real
and imaginary parts of the complex variable $\bz$; see~\cite{Kre09}.
Then, two gradients are defined based on the partial derivatives of the
function $f$ with respect to the real and imaginary parts of the variable $\bz$:
the Wirtinger derivative denoted by
$\frac{\diff}{\diff \bz} f(\bz, \overline{\bz})$ and the
conjugate Wirtinger derivative denoted as
$\frac{\diff}{\diff \overline{\bz}}f(\bz,\overline{\bz})$;
see~\cite{Wir27}.
In the case that the Wirtinger derivative or conjugate Wirtinger derivative is
$\zeros$, the function $f$ is at a stationary point.
Furthermore, since in this work we differentiate only real-valued error
functions, the following identity holds
\begin{equation}
  \frac{\diff f(\bz,\overline{\bz})}{\diff \overline{\bz}} =
    \overline{\frac{\diff f(\bz,\overline{\bz})}{\diff \bz}}.
\end{equation}
Consequently, it is sufficient for us to compute only the Wirtinger derivatives
of our error functions in this work.
For further information on the Wirtinger calculus, we refer the reader to the
collection of lecture notes~\cite{Kre09}.


\subsubsection{Wirtinger derivatives of error functions}%
\label{sec:graderr}

With the Wirtinger calculus at hand, we provide in this section the gradients
of the nonlinear rational least-squares error criterion~\cref{eqn:L2Err_AAA}
and the approximations to that least-squares error done
in~\cref{eqn:LevyErrCrit} (Levy),~\cref{eqn:SKApx} (\SK),
and~\cref{eqn:WFError} (\WF).
By comparison of the different gradients, we will see that only \WF{}
upon convergence truly minimizes the objective of
interest~\cref{eqn:L2Err_AAA}.
However, the gradients of the Levy approximation and \SK{} will reveal in which
cases the results of these two approximations are \emph{nearly} optimal.

We begin with the Wirtinger derivative of the nonlinear least-squares error
in the following lemma.

\begin{lemma}
  The Wirtinger derivative of the nonlinear rational least-squares
  error~\cref{eqn:L2Err_AAA} with respect to the barycentric weights
  $\w$ is given by
  \begin{equation} \label{eqn:grad_ltwo}
    \dEdw = \sum_{i = 1}^{M} \frac{1}{d(z_{i}; \w)}
      \big( \nBasis(z_{i}) - r(z_{i}; \w) \dBasis(z_{i}) \big)
      \overline{\big( r(z_{i}; \w) - H(z_{i}) \big)}.
  \end{equation}
\end{lemma}
\begin{proof}
  The result directly follows from applying the partial derivative
  formula~\cref{eqn:totalDerivative} to the nonlinear rational least-squares
  error~\cref{eqn:L2Err_AAA}.
\end{proof}

The derivative given in~\cref{eqn:grad_ltwo} needs to vanish at the stationary
points of the nonlinear rational least-squares error.
In contrast to that, the following theorem provides the Wirtinger derivatives
for the approximations used in the different solution procedures from above,
namely Levy, \SK{}, and \WF{}.

\begin{theorem}
  The Wirtinger derivative of the Levy approximation error
  criterion~\cref{eqn:LevyErrCrit} with respect to the barycentric weights
  $\w$ is given by
  \begin{equation} \label{eqn:grad_levy}
    \dLevydw = \sum_{i = 1}^{M}
      \big( \nBasis(z_{i}) - H(z_{i}) \dBasis(z_i) \big)
    \overline{\big(n(z_{i}; \w) - d(z_{i}; \w) H(z_{i}) \big)}.
  \end{equation}
  In the case that the \SK{} iteration converges, the Wirtinger derivative of
  the \SK{} error criterion~\cref{eqn:SKApx} in the corresponding fixed point
  with respect to the barycentric weights $\w$ is given by
  \begin{equation} \label{eqn:grad_SK}
    \dSKdw = \sum_{i = 1}^{M} \frac{1}{d(z_{i}; \w)}
      \big( \nBasis(z_{i}) - H(z_{i}) \dBasis(z_{i}) \big)
      \overline{\big( r(z_{i}; \w) - H(z_{i}) \big)}.
  \end{equation}
  Finally, if the \WF{} iteration converges, then the Wirtinger derivative of
  the \WF{} error criterion~\cref{eqn:WFError} in the corresponding fixed point
  with respect to the barycentric weights $\w$ is given by
  \begin{equation} \label{eqn:grad_WF}
    \dWFdw = \sum_{i = 1}^{M} \frac{1}{d(z_{i}; \w)}
      \big( \nBasis(z_{i}) - r(z_{i}; \w) \dBasis(z_{i}) \big)
      \overline{\big( r(z_{i}; \w) - H(z_{i}) \big)}.
  \end{equation}
\end{theorem}
\begin{proof}
  Starting with the Levy approximation in~\cref{eqn:LevyErrCrit}, we can rewrite
  this expression as
  \begin{equation} \label{eqn:sep_conj}
    \ELevy = \sum_{i = 1}^{M} \big( n(z_{i}; \w) - d(z_{i}; \w) H(z_{i}) \big)
      \overline{\big( n(z_{i}; \w) - d(z_{i}; \w) H(z_{i}) \big)}.
  \end{equation}
  Computing the Wirtinger derivative of this with respect to $\w$ only affects
  the first term in the product, which results in~\cref{eqn:grad_levy}.
  For the \SK{} iteration, we have in the $p$-th step the error expression
  \begin{equation} \label{eqn:sk_perr}
    \Esk = \sum_{i = 1}^{M}
      \frac{1}{\lvert d(z_{i}; \w\skwfIter[p - 1]) \rvert^{2}}
      \lvert n(z_{i}; \w) - d(z_{i}; \w) H(z_{i}) \rvert^{2}.
  \end{equation}
  We note that the variable weights only appear in the absolute value of the
  error but not in the pre-multiplied fraction.
  As in~\cref{eqn:sep_conj}, we can split the absolute value into the
  multiplication with the conjugate expression and, as a result, the Wirtinger
  derivative of~\cref{eqn:sk_perr} is given by
  \begin{equation} \label{eqn:grad_SKnotConverged}
    \dSKdw = \sum_{i = 1}^{M}
      \frac{1}{\lvert d(z_{i}; \w\skwfIter[p - 1]) \rvert^{2}}
      \big( \nBasis(z_{i}) - H(z_{i}) \dBasis(z_{i}) \big)
      \overline{\big( n(z_{i}; \w) - H(z_{i}) d(z_{i}; \w) \big)}.
  \end{equation}
  In the case that the \SK{} iteration converges, we have that
  $\w\skwfIter[p-1] = \w\skwfIter$.
  Then, the expression in~\cref{eqn:grad_SKnotConverged} can be simplified
  into~\cref{eqn:grad_SK}.
  Finally, for the \WF{} iteration in the $p$-th step, we have the error
  expression
  \begin{equation} \label{eqn:wf_perr}
    \begin{aligned}
      \Ewf = \sum_{i = 1}^{M} &
        \frac{1}{\lvert d(z_{i}; \w\skwfIter[p - 1]) \rvert^{2}}
        \lvert n(z_{i}; \w) - r(z_{i}; \w\skwfIter[p - 1]) d(z_{i}; \w) \\
      & {}+{}
        n(z_{i}; \w\skwfIter[p - 1]) - d(z_{i}; \w\skwfIter[p - 1])
        H(z_{i}) \rvert^{2},
    \end{aligned}
  \end{equation}
  where $\w\skwfIter[p - 1]$ are as in \SK{} the weights obtained in the
  previous step and independent of the variable $\w$.
  Using the same rewriting of the absolute value as in~\cref{eqn:sep_conj},
  the Wirtinger derivative of~\cref{eqn:wf_perr} is given by
  \begin{equation} \label{eqn:grad_WFnotConverged}
    \begin{aligned}
      \dWFdw = \sum_{i = 1}^{M} &
        \frac{1}{\lvert d(z_{i}; \w\skwfIter[p - 1]) \rvert^{2}}
        \big( \nBasis(z_{i}) - r(z_{i}; \w\skwfIter[p - 1]) \dBasis(z_{i}) \big)
        \overline{\big( n(z_{i}; \w)} \\
        &\overline{{}-{}
        r(z_{i}; \w\skwfIter[p - 1]) d(z_{i}; \w) +
        n(z_{i}; \w\skwfIter[p - 1]) - d(z_{i}; \w\skwfIter[p - 1])
        H(z_{i}) \big)}.
    \end{aligned}
  \end{equation}
  In the case that the \WF{} iteration converges so that
  $\w\skwfIter[p-1] = \w\skwfIter$, the derivative
  in~\cref{eqn:grad_WFnotConverged} can be rewritten as~\cref{eqn:grad_WF}.
  This concludes the proof.
\end{proof}

It can immediately be verified that the gradients
in~\cref{eqn:grad_ltwo,eqn:grad_WF} are identical.
Thus, if the \WF{} iteration converges, it returns a stationary point of the
rational least-squares error with respect to the barycentric weights $\w$
as intended by the iteration.
The same cannot be said for the Levy approximation and \SK{}.
In the following, we will compare these gradients with the one for the
rational least-squares problem to provide further insights into these
approximations.


\subsubsection{Analysis of error derivatives for Levy and SK}%
\label{sec:gradAnalysis}

As the expression in~\cref{eqn:grad_SK} is visibly very similar
to~\cref{eqn:grad_ltwo}, we begin our analysis with the \SK{} iteration.
Comparing the two expressions reveals the only difference being that 
in~\cref{eqn:grad_ltwo}, the vector of denominator basis functions
$\dBasis(z_i)$ is weighted by the rational function $r(z_{i}; \w)$, while
in~\cref{eqn:grad_SK} it is weighted by the given data $H(z_{i})$.
Hence, we expect that if the \SK{} iteration results in a rational approximant
$r(z; \w)$, which approximates the given data in $\dataSet$ well in the sense
that
\begin{equation} \label{eqn:ConditionForSKGood}
  \lvert r(z_{i}; \w) - H(z_{i}) \rvert ~\text{is small,} \quad
  \text{for}~i = 1, \ldots, M,
\end{equation}
then $r(z; \w)$ should be nearly optimal in the true, rational
error~\cref{eqn:L2Err_AAA}.
Conversely, if $r(z; \w)$ is an insufficient approximation to the data in
$\dataSet$, i.e., it produces a large approximation error, we would expect 
$r(z; \w)$ to be further away from optimality in the nonlinear
error~\cref{eqn:L2Err_AAA}.

Let us consider these observations in the context of numerical procedures that
construct rational approximations such as \Cref{alg:AAA,alg:AAAA}.
In early iterations, when the degree of the rational approximant is small
and large approximation errors are expected, we would not expect the \SK{}
iteration to yield close-to-optimal approximations.
On the other hand, for larger degree rational approximations, when the
approximation errors are expected to be small so
that~\cref{eqn:ConditionForSKGood} is likely to hold,
we would expect the \SK{} iteration to yield near-optimal approximation results.
Furthermore, we note that these observations yield a strong motivation for the
greedy interpolation step in both \Cref{alg:AAA,alg:AAAA}.
Thereby, the worst case approximation error of the rational approximation
is set to zero in every iteration step, which is highly advantageous for
satisfying the condition~\cref{eqn:ConditionForSKGood}.

Proceeding to the Levy approximation, we see that after some algebraic
manipulations, the derivative in~\cref{eqn:grad_levy} can be rewritten as
\begin{equation} \label{eqn:grad_levyRearanged}
  \dLevydw = \sum_{i = 1}^{M} \lvert d(z_{i}; \w) \rvert^{2}
    \left( \frac{1}{d(z_{i}; \w)} \big( \nBasis(z_{i}) - H(z_{i})
    \dBasis(z_{i}) \big) \overline{\big( r(z_{i}; \w) - H(z_{i}) \big)}\right).
\end{equation}
Comparing~\cref{eqn:grad_levyRearanged} to~\cref{eqn:grad_ltwo} reveals two
major differences:
First, the terms in~\cref{eqn:grad_levyRearanged} are weighted by the
denominator $\lvert d(z_{i}; \w) \rvert^{2}$ and
second, as previously for \SK{}, the vector of denominator basis functions
$\dBasis(z_{i})$ is weighted by $H(z_{i})$ rather than $r(z_{i}; \w)$.

Following the latter point, parts of the analysis for the \SK{} iteration also
hold for the Levy approximation.
If the weights $\w$ obtained via the Levy approximation result in a poor
approximation to the data in $\dataSet$, then we expect that the rational
$r(z; \w)$ approximation is far from optimal.
Beyond that, the weighting by $\lvert d(z_{i}; \w) \rvert^{2}$
in~\cref{eqn:grad_levyRearanged} acts as another error source concerning
optimality, which we now examine in detail.

Let $\w_{\ast}$ be the solution to the linear least-squares
problem~\cref{eqn:LevyL2Prob_Mat}, in other words, the weights that minimize the
Levy error criterion~\cref{eqn:LevyErrCrit}.
In the case that $\lvert d(z_{i}; \w_{\ast}) \rvert \equiv d \in \C$, for
$i = 1, 2, \ldots, M$, then we have that
\begin{equation}
  \zeros = \left. \dLevydw \right\rvert_{\w = \w_{\ast}} =
    \left. M \lvert d \rvert^{2} \dSKdw \right|_{\w = \w_{\ast}}.
\end{equation}
Hence, when $\lvert d(z_{i}; \w_{\ast}) \rvert$ is constant, then minimizers of the Levy
approximation are also the minimizers in the \SK{} iteration and both approaches
provide the same level of accuracy.
However, in the case that $\lvert d(z_{i}; \w_{\ast}) \rvert$ is not constant but
varies by several orders of magnitude over the data $\dataSet$, we expect the
minimizers of the Levy approximation to be far from the optimal approximation
in the rational least-squares error~\cref{eqn:L2Err_AAA}.

We are now in a position to justify the differences in the performances
of \AAA{} and \AAAA{} for the two example functions in \Cref{fig:IntroExs}.
As seen in \Cref{fig:ReluXEx}, \AAA{} struggles at iteration $k = 14$
to approximate $\relu(x)$ while \AAAA{} already provides a normalized $\ltwo$
error of less than $10^{-5}$.
In contrast, both \AAA{} and \AAAA{} are capable of constructing suitably
accurate approximations to $\lvert x \rvert$ at iteration $k = 14$;
see~\cref{fig:absXEx}.
As suggested by our analysis, we examine the variation in the denominator of
function $d(z; \w)$ constructed via \AAA{} for $\lvert x \rvert$ and $\relu(x)$
at iteration $14$ over the real approximation interval $[-1, 1]$ shown in
\Cref{fig:IntroExsDenoms}.
In \Cref{fig:absXExDenom}, we see that while the denominator of the \AAA{}
approximation to $\lvert x \rvert$ is not constant, its variation is constrained
to only two orders of magnitude.
However, on the other side, we see in \Cref{fig:ReluXExDenom} that the
denominator of the \AAA{} approximation for $\relu(x)$ shows nearly $17$ orders
of magnitude in variation.
Correspondingly, the \AAA{} approximation to $\lvert x \rvert$ is nearly as
good as the locally optimal \AAAA{} approximation in \Cref{fig:absXEx}, and the
\AAA{} approximation to $\relu(x)$ is clearly not close to optimal in
\Cref{fig:ReluXEx}.

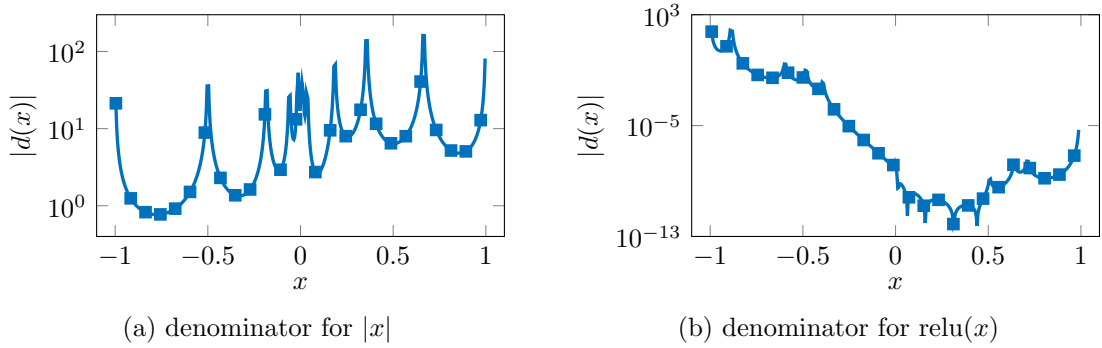
\begin{figure}[t]
  \centering
  \begin{subfigure}[b]{.49\linewidth}
    \centering
  \tikzexternalenable%
  \tikzsetnextfilename{absXDenom}%
  \begin{tikzpicture}[font = \plotfontsize]
  \pgfplotstableread{graphicsPaper/data/introEx_denom_h.csv}\tableERR

  \begin{semilogyaxis}[
    scale only axis,
    width              = .73\linewidth,
    height             = .4\linewidth,
    xmin               = -1.1,
    xmax               = 1.1,
    ymin               = 4e-1,
    ymax               = 3e+2,
    xminorticks        = false,
    yminorticks        = false,
    scaled x ticks     = false,
    xlabel             = {$x$},
    ylabel             = {$\lvert d(x) \rvert$},
    xlabel style       = {yshift = .3em},
    ylabel style       = {yshift = -.3em},
    x tick label style = {/pgf/number format/1000 sep={\,}},
    y tick label style = {/pgf/number format/1000 sep={\,}}
  ]
    \addplot[AAAConverge, mark repeat = {20}] table[x = mu, y = denom]{\tableERR};
  \end{semilogyaxis}
\end{tikzpicture}%
  \tikzexternaldisable%

    \caption{denominator for $\lvert x \rvert$}
    \label{fig:absXExDenom}
  \end{subfigure}%
  \hfill%
  \begin{subfigure}[b]{.49\linewidth}
    \centering
  \tikzexternalenable%
  \tikzsetnextfilename{reluXDenom}%
  \begin{tikzpicture}[font = \plotfontsize]
  \pgfplotstableread{graphicsPaper/data/introEx_denom_g.csv}\tableERR

  \begin{semilogyaxis}[
    scale only axis,
    width              = .73\linewidth,
    height             = .4\linewidth,
    xmin               = -1.1,
    xmax               = 1.1,
    ymin               = 1e-13,
    ymax               = 1e+3,
    xminorticks        = false,
    yminorticks        = false,
    scaled x ticks     = false,
    xlabel             = {$x$},
    ylabel             = {$\lvert d(x) \rvert$},
    xlabel style       = {yshift = .3em},
    ylabel style       = {yshift = -.3em},
    x tick label style = {/pgf/number format/1000 sep={\,}},
    y tick label style = {/pgf/number format/1000 sep={\,}}
  ]
    \addplot[AAAConverge, mark repeat = {20}] table[x = mu, y = denom]{\tableERR};
  \end{semilogyaxis}
\end{tikzpicture}%
  \tikzexternaldisable%

    \caption{denominator for $\relu(x)$}
    \label{fig:ReluXExDenom}
  \end{subfigure}
  
  \caption{Denominator functions of degree $k = 14$ of the \AAA{}
    approximations to $\lvert x \rvert$ and $\relu(x)$ on the
    real interval $[-1, 1]$:
    While the variation of the approximating denominator for $\lvert x \rvert$
    is constrained to only two order of magnitude, the denominator varies
    nearly $17$ orders of magnitude in the approximation of $\relu(x)$.}
  \label{fig:IntroExsDenoms}
\end{figure}


\subsection{Initializations for Whitfield's iteration}
\label{sec:WFInits}

As with all nonlinear optimization methods, a good initialization is essential
for the success of the method.
In this section, we detail our initialization strategy for the \WF{} iteration
in the \AAAA{} algorithm, which ensures monotonic error decay.

It has been observed in~\cite{DesDA06} that in some situations, initializing
the \WF{} iteration with the output of the \SK{} iteraiton can lead to better
convergence of \WF{}.
This is supported by our previous analysis of the gradients in
\Cref{sec:gradAnalysis} in so far that the \SK{} iteration may provide nearly
optimal results in the case of small approximation errors.
We note that since the first step of the \SK{} iteration computes the Levy
approximation, this initialization strategy also implicitly considers the use
of the Levy approximation, in particular when the iterate of \SK{} with the
smallest approximation error is used for the initialization.
While this may work well, we have also observed in several numerical instances
that the weights $\w\AAAiter\skwfIter$ computed during all the \SK{} iteration
steps only increased the nonlinear least-squares error so that
\begin{equation}
  \sum_{i = 1}^{M} \left\lvert H(z_{i}) -
    \frac{\left( \w\AAAiter[k-1] \right)^{\trans} \nBasis\AAAiter[k-1](z_{i})}
    {\left( \w\AAAiter[k-1] \right)^{\trans} \dBasis\AAAiter[k-1](z_{i})}
    \right\rvert^{2} <
  \sum_{i = 1}^{M} \left\lvert H(z_{i}) -
    \frac{\left( \w\AAAiter\skwfIter \right)^{\trans} \nBasis\AAAiter(z_{i})}
    {\left( \w\AAAiter\skwfIter \right)^{\trans} \dBasis\AAAiter(z_{i})}
    \right\rvert^{2},
\end{equation}
for all $p = 0, 1, \ldots, \maxOrderSK$.
To avoid the initialization of \WF{} with such an unintended choice of weights,
we propose to use the following initialization vector in the cases when the
\SK{} iteration does not yield any suitable results:
\begin{equation} \label{eqn:zeroWFInit}
  \bmat{\left(\w\AAAiter[k-1]\right)^{\trans} & 0}^{\trans}.
\end{equation}
These are the weights from the previous \AAAA{} step appended by $0$.
Thus, the initialization~\cref{eqn:zeroWFInit} begins the \WF{} iteration
with a rational function that is equal to the final rational approximation from
the previous iteration step.

While we have observed that one of these two initialization
strategies lead to a suitable decrease in the approximation error in most cases,
it is still potentially possible that the \WF{} iteration does not decrease the
error at any iteration.
In this case, in order to preserve the error monotonicity of the \AAAA{}
method, we set the weights to
\begin{equation} \label{eqn:ChoseZeroWeight}
  \w\AAAiter = \bmat{\left( \w\AAAiter[k-1] \right)^{\trans} & 0}^{\trans}.
\end{equation}
While this choice does preserve monotonicity, it naturally leads to a problem
in the greedy interpolation step, which can cause stagnation of
the \AAAA{} algorithm.
Specifically, if the weights at iteration $k$ are chosen according
to~\cref{eqn:ChoseZeroWeight}, then we are effectively eliminating the
contribution of the $k$-th interpolation point from the error behavior
of the approximation.
In typical scenarios, this leads to the next support point $\lambda_{k + 1}$
to be chosen adjacent or at least close to the previous support
point $\lambda_{k}$ such that $\dBasis\AAAiter[k+1]$ and $\nBasis\AAAiter[k+1]$
are not much more expressive than $\dBasis\AAAiter$ and $\nBasis\AAAiter$
leading to the stagnation of \AAAA{}.
The key to overcoming this issue is to change the error measure for the greedy
selection step whenever the weights are chosen as~\cref{eqn:ChoseZeroWeight}. 
We propose two options here.
First, we recommend switching the original deterministic greedy selection
\begin{equation} \label{eqn:probGreedy}
  (\lambda_{k}, h_{k}) = \argmax\limits_{(z_{i}, H(z_{i})) \in \dataSet}
    \lvert r(z_{i}; \w\AAAiter[k-1]) - H(z_{i}) \rvert
\end{equation}
to a probabilistic greedy selection.
Thereby, we sample $(\lambda_{k}, h_{k})$ from a probability distribution
that is proportional to the approximation error
$\lvert r(z_{i}; \w\AAAiter[k-1]) - H(z_{i}) \rvert$, for
$(z_{i}, H(z_{i})) \in \dataSet$ and $i = 1, 2, \ldots, M$.
A second alternative that we recommend is to change to a relative error
measure, that is, we choose $(\lambda_{k}, h_{k})$ deterministically via
\begin{equation} \label{eqn:relGreedy}
  (\lambda_{k}, h_{k}) = \argmax\limits_{(z_{i}, H(z_{i})) \in \dataSet,\,
    H(z_{i}) \neq 0}
    \frac{\lvert r(z_{i}; \w\AAAiter[k-1]) - H(z_{i}) \rvert}%
    {\lvert H(z_{i}) \rvert}.
\end{equation}
We emphasize that we only use either of these two alternative greedy
selection strategies if the weights $\w\AAAiter$ are chosen according
to~\cref{eqn:ChoseZeroWeight}, i.e., the Levy approximation, \SK{} iteration, and
\WF{} iteration all failed to decrease the rational approximation error.
Only in this scenario, we replace \Cref{alg:AAAA_max} of \Cref{alg:AAAA} with
either the probabilistic greedy step~\cref{eqn:probGreedy} or the relative
error greedy step~\cref{eqn:relGreedy}.
While in general the performance of these alternatives to the classical greedy
step depends on the problem, we have observed that in all considered examples
either method is sufficient to prevent the stagnation of \AAAA{}.


\section{Numerical experiments}%
\label{sec:numerics}

In this section, we demonstrate the performance of our proposed \AAAA{} method
in comparison to the classical \AAA{} on both generic functions and
reduced-order modeling benchmark problems.
The experiments reported here were performed on a 2023 MacBook Pro equipped
with 16\,GB RAM and an Apple M2 Pro chip.
Computations were done in MATLAB 25.1.0.2973910 (R2025a) Update 1 running on
macOS Sequoia 15.6.1.
The source codes, data and results of the numerical experiments are available
at~\cite{supAckBW26}.


\subsection{Experimental setup}%
\label{sec:setup}

For each of the examples presented below, we have generated data sets of
the form
\begin{equation}
  \dataSet = \big\{ (z_{1}, H(z_{1})), (z_{2}, H(z_{2})), \ldots,
    (z_{M}, H(z_{M})) \big\},
\end{equation}
where $H\colon \C \to \C$ denotes the original function that we aim to
approximate.

For the comparison of the \AAA{} and \AAAA{} methods in terms of
approximation quality, we compute and plot normalized versions of the discrete
$\ltwo$ and $\linf$ error measures over the respective data set
$\dataSet$ given by
\begin{equation}
  \frac{\left(\displaystyle \sum\limits_{i = 1}^{M}
    \lvert H(z_{i}) - r(z_{i}) \rvert^{2}\right)^{1/2}}%
    {\left(\displaystyle \sum\limits_{i = 1}^{M} \lvert H(z_{i}) \rvert^{2}\right)^{1/2}}
  \quad\text{and}\quad
  \frac{\displaystyle \max\limits_{(z, H(z)) \in \dataSet}
    \lvert H(z) - r(z) \rvert}%
    {\displaystyle \max_{(z, H(z)) \in \dataSet} \lvert H(z) \rvert}.
\end{equation}

For \AAAA{}, in the steps in which the refinement strategies
(\SK{} and \WF{} iterations) fail to decrease the approximation error,
we are changing to the randomized greedy selection as outlined in
\Cref{sec:WFInits}.
The supplemental code~\cite{supAckBW26} also provides the option to change
to the relative error measures as the alternative support point selection
criterion.


\subsection{Classical function approximations}%
\label{sec:genFuncs}

As the first test of our proposed method, we aim to approximate two
more generic nonlinear functions, namely $\lvert \sin(3 \pi x) \rvert$ and
$\triWave(x)$ (defined below),
via rational functions similar to the motivating examples presented in
\Cref{sec:intro}.
These two functions are chosen as they distill certain features, which
lead to unsatisfactory convergence behavior of the classical \AAA{} method.
On the other hand, the proposed \AAAA{} does not seem to be negatively affected
by these features.
Furthermore, these functions include sharp, non-differentiable edges that can
be also seen in model order reduction benchmarks; see \Cref{sec:ROMBenchmarks}
for further details.

\begin{figure}[t]
  \centering
  \begin{subfigure}[b]{.49\linewidth}
    \centering
  \tikzexternalenable%
  \tikzsetnextfilename{AbsSin_converge}%
  \begin{tikzpicture}[font = \plotfontsize]
  \pgfplotstableread{graphicsPaper/data/genericFuns_conv_f1.csv}\tableERR

  \begin{semilogyaxis}[
    scale only axis,
    width              = .73\linewidth,
    height             = .4\linewidth,
    xmin               = 0,
    xmax               = 50,
    ymin               = 1e-4,
    ymax               = 2e+2,
    xminorticks        = false,
    yminorticks        = true,
    scaled x ticks     = false,
    xlabel             = {degree $k$},
    ylabel             = {normalized $\ltwo$ error},
    xlabel style       = {yshift = .3em},
    ylabel style       = {yshift = -.3em},
    x tick label style = {/pgf/number format/1000 sep={\,}},
    y tick label style = {/pgf/number format/1000 sep={\,}}
  ]
    \addplot[AAAConverge] table[x = iter, y = AAA]{\tableERR};
    \addplot[AAAAConverge] table[x = iter, y = NLAAA]{\tableERR};
  \end{semilogyaxis}
\end{tikzpicture}%
  \tikzexternaldisable%

    \caption{convergence for $\lvert \sin(3 \pi x) \rvert$}
    \label{fig:absSin_conv}
  \end{subfigure}%
  \hfill%
  \begin{subfigure}[b]{.49\linewidth}
    \centering
  \tikzexternalenable%
  \tikzsetnextfilename{triWave_converge}%
  \begin{tikzpicture}[font = \plotfontsize]
  \pgfplotstableread{graphicsPaper/data/genericFuns_conv_f2.csv}\tableERR

  \begin{semilogyaxis}[
    scale only axis,
    width              = .73\linewidth,
    height             = .4\linewidth,
    xmin               = 0,
    xmax               = 50,
    ymin               = 5e-3,
    ymax               = 6e+1,
    xminorticks        = false,
    yminorticks        = true,
    scaled x ticks     = false,
    xlabel             = {degree $k$},
    ylabel             = {normalized $\ltwo$ error},
    xlabel style       = {yshift = .3em},
    ylabel style       = {yshift = -.3em},
    x tick label style = {/pgf/number format/1000 sep={\,}},
    y tick label style = {/pgf/number format/1000 sep={\,}}
  ]
    \addplot[AAAConverge] table[x = iter, y = AAA]{\tableERR};
    \addplot[AAAAConverge] table[x = iter, y = NLAAA]{\tableERR};
  \end{semilogyaxis}
\end{tikzpicture}%
  \tikzexternaldisable%

    \caption{convergence for $\triWave(x)$}
    \label{fig:trieWave_conv}
  \end{subfigure}%

  \vspace{.5\baselineskip}
  \begin{subfigure}[b]{.49\linewidth}
    \centering
  \tikzexternalenable%
  \tikzsetnextfilename{AbsSin_plot}%
  \begin{tikzpicture}[font = \plotfontsize]
  \pgfplotstableread{graphicsPaper/data/genericFuns_eval_f1.csv}\tableRESP

  \begin{axis}[
    scale only axis,
    width              = .73\linewidth,
    height             = .4\linewidth,
    xmin               = -1,
    xmax               = 1,
    ymin               = -0.1,
    ymax               = 1.1,
    xminorticks        = true,
    yminorticks        = true,
    scaled x ticks     = false,
    clip mode          = individual,
    xlabel             = {$x$},
    ylabel             = {$r(x)$},
    xlabel style       = {yshift = .3em},
    ylabel style       = {yshift = -.3em},
    x tick label style = {/pgf/number format/1000 sep={\,}},
    y tick label style = {/pgf/number format/1000 sep={\,}}
  ]
    \addplot[trueData] table[x = mu, y = g]{\tableRESP};
    \addplot[AAAResponse] table[x = mu, y = AAA]{\tableRESP};
    \addplot[AAAAResponse] table[x = mu, y = NLAAA]{\tableRESP};
  \end{axis}
\end{tikzpicture}%
  \tikzexternaldisable%

    \caption{degree-$50$ approximations for $\lvert \sin(3 \pi x) \rvert$}
    \label{fig:AbsSin_plot}
  \end{subfigure}%
  \hfill%
  \begin{subfigure}[b]{.49\linewidth}
    \centering
  \tikzexternalenable%
  \tikzsetnextfilename{triWave_plot}%
  \begin{tikzpicture}[font = \plotfontsize]
  \pgfplotstableread{graphicsPaper/data/genericFuns_eval_f2.csv}\tableRESP

  \begin{axis}[
    scale only axis,
    width              = .73\linewidth,
    height             = .4\linewidth,
    xmin               = -1,
    xmax               = 1,
    ymin               = -0.1,
    ymax               = 1.1,
    xminorticks        = true,
    yminorticks        = true,
    scaled x ticks     = false,
    clip mode          = individual,
    xlabel             = {$x$},
    ylabel             = {$r(x)$},
    xlabel style       = {yshift = .3em},
    ylabel style       = {yshift = -.3em},
    x tick label style = {/pgf/number format/1000 sep={\,}},
    y tick label style = {/pgf/number format/1000 sep={\,}}
  ]
    \addplot[trueData] table[x = mu, y = g]{\tableRESP};
    \addplot[AAAResponse] table[x = mu, y = AAA]{\tableRESP};
    \addplot[AAAAResponse] table[x = mu, y = NLAAA]{\tableRESP};
  \end{axis}
\end{tikzpicture}%
  \tikzexternaldisable%

    \caption{degree-$50$ approximations $\triWave(x)$}
    \label{fig:triWave_plot}
  \end{subfigure}%
  
  \vspace{.5\baselineskip}
  \tikzexternalenable%
  \tikzsetnextfilename{LegendResults}%
  \begin{tikzpicture}
  \begin{axis}[%
    hide axis,
    scale only axis,
    width  = 1cm,
    height = 1cm,
    xmin   = 0,
    xmax   = 1,
    ymin   = 0,
    ymax   = 1,
    legend columns    = -1,
    legend cell align = {left},
    legend style      = {
      at     = {(0,0)},
      anchor = center,
      /tikz/every even column/.append style = {column sep = 0.0cm}}
  ]
    \addlegendimage{trueData} coordinates {(0, 0)};
    \addlegendentry{data\phantom{Pp}}

    \addlegendimage{AAAConverge}
    \addlegendentry{/}
    \addlegendimage{AAAResponse}
    \addlegendentry{\AAA{}\phantom{Pp}}

    \addlegendimage{AAAAConverge}
    \addlegendentry{/}
    \addlegendimage{AAAAResponse}
    \addlegendentry{\AAAA}
  \end{axis}
\end{tikzpicture}%
  \tikzexternaldisable%

  \caption{Convergence and approximation results of \AAA{} and \AAAA{} for the
    functions $\lvert \sin(3\pi x) \rvert$ and $\triWave(x)$:
    For both example functions, the classical \AAA{} has a very unsteady error
    convergence while the proposed \AAAA{} provides a monotonic decay.
    For the $\triWave$ function, the final approximation by \AAA{} shows visible
    deviations from the given data while \AAAA{} provides an indistinguishable
    approximation.}
  \label{fig:ConvAndErr_generic}
\end{figure}
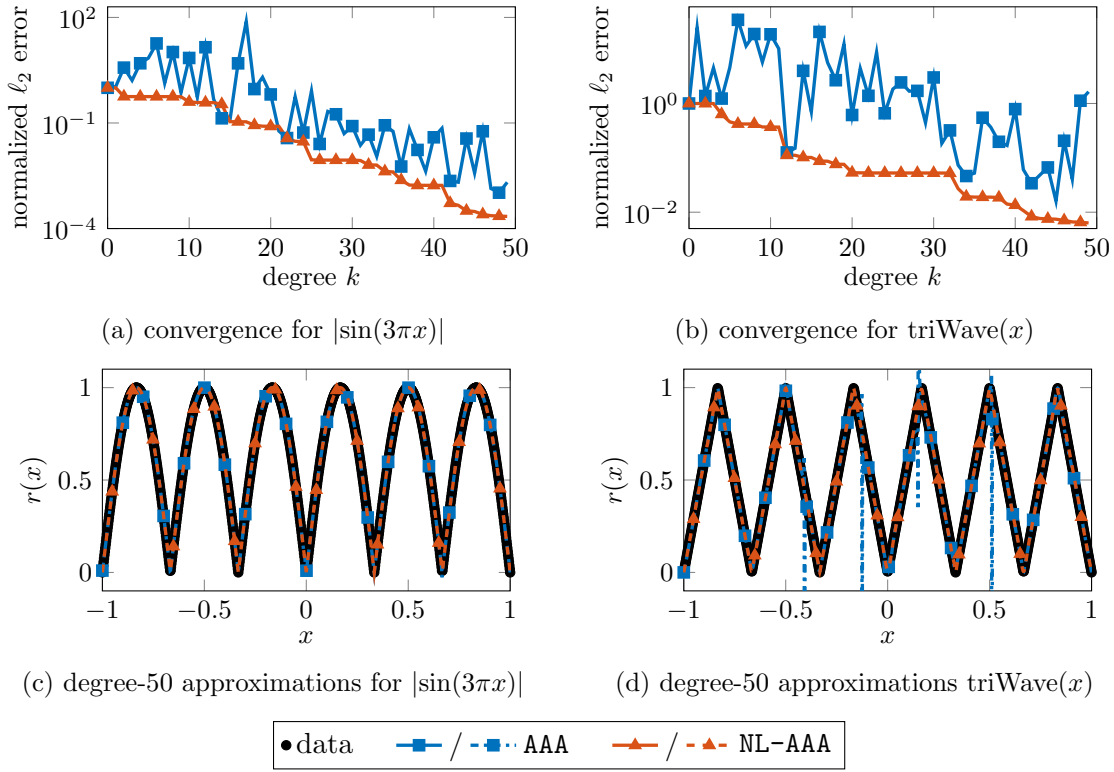

Let us first consider the approximation of $\lvert \sin(3 \pi x) \rvert$
in the real interval $x \in [-1, 1]$.
As data, we have sampled the function in $1\,000$ real, linearly equidistant points
in the considered interval and used this data to compute rational
approximations via \AAA{} and \AAAA{} up to degree $50$.
The convergence of the two methods in the normalized (squared) $\ltwo$-norm
is shown in \Cref{fig:absSin_conv}.
We can see that while the error of the \AAA{} approximation tends downwards,
it is far from monotonic and has large variations in the magnitude of the
error that can be observed as spikes in the plot.
In contrast, the \AAAA{} approximation error monotonically decreases over the
iterations and has smaller errors than \AAA{} in nearly every step.
For degree $50$, both methods yield a suitable approximation
to the given function $f(x) = \lvert \sin(3 \pi x) \rvert$  as can be seen
in \Cref{fig:AbsSin_plot}.

In general, we have observed that \AAA{} tends to struggle with its error
convergence and approximation quality when the function underlying the given
data is less smooth and has points at which the derivative changes rapidly
or does not exist.
Another function of that type is the triangular wave.
This function is defined on the interval $x \in [-1, 1]$ via
\begin{equation}
  \triWave(x) = 2 \left\lvert 3 x - 
    \left\lfloor 3 x + \frac{1}{2} \right\rfloor \right\rvert,
\end{equation}
where $\lfloor . \rfloor$ denotes the integer floor function.
In contrast to the previous function $\lvert \sin(3 \pi x) \rvert$, the
triangular wave has double the amount of non-differentiable points in the
interval~$[-1, 1]$.
This particular triangular wave can also be obtained via the piecewise linear
interpolation of $\lvert \sin(3 \pi x) \rvert$.

As for the previous function, we sample $\triWave(x)$ in $1\,000$ real, linearly
equidistant points in the interval $[-1, 1]$ to construct our data, and then
we use \AAA{} and \AAAA{} to construct rational approximations up to
degree $50$.
In \Cref{fig:trieWave_conv}, we see that this time the errors of \AAA{}
and \AAAA{} lie even further apart.
As for the previous example, we see that the error behavior of \AAA{} yields
many rapid changes.
However, this time no general convergence trend is visible and most of the
approximations constructed via \AAA{} up to degree $50$ yield a normalized
$\ltwo$ error that is larger than $1$.
The degree-$50$ approximations in \Cref{fig:triWave_plot} show that in several
points, the \AAA{} approximation strongly deviates from the given data.
On the other hand, the proposed \AAAA{} algorithm provides again a monotonic
error decay and the degree-$50$ approximation is indistinguishable from the
given data.


\subsection{Model order reduction benchmark examples}%
\label{sec:ROMBenchmarks}

The previous numerical examples have indicated that functions with sharp edges
and non-differentiable points pose a larger challenge for the classical \AAA{}
algorithm compared to our proposed \AAAA{} method.
For many model order reduction examples, similar behaviors can be observed with
data given close to non-differentiable points with quickly varying function
values.
In the following, we compare the \AAA{} and \AAAA{} algorithms on two
established model order reduction examples from the literature.
As we consider here the underlying dynamical systems rather than the
corresponding rational functions, we denote the complexity of the results by
the order of the underlying model, which is the degree of the corresponding
rational transfer function plus one.


\subsubsection{International space station}%
\label{sec:rom_iss}

The first model order reduction example we consider is a dynamical system model
of stage 12A of the international space station
(ISS)~\cite{GugAB01, morwiki_iss}.
The transfer function corresponding to the dynamical system is a
matrix-valued rational function of degree $1\,412$.
As we consider scalar functions in this work, we only take the upper left
entries of these matrix-valued transfer functions.
This models the input-to-output behavior of the system from the first system
input to the first observed system output.
To generate data for this example, we sample the transfer function at $1\,000$
logarithmically equidistant points on the imaginary acis in the interval
$[10^{-1}, 10^{3}] \imunit$.
This range captures the main behavior of the system's frequency response.
The magnitudes of the sampled data points are shown in \Cref{fig:ISS_resp}.

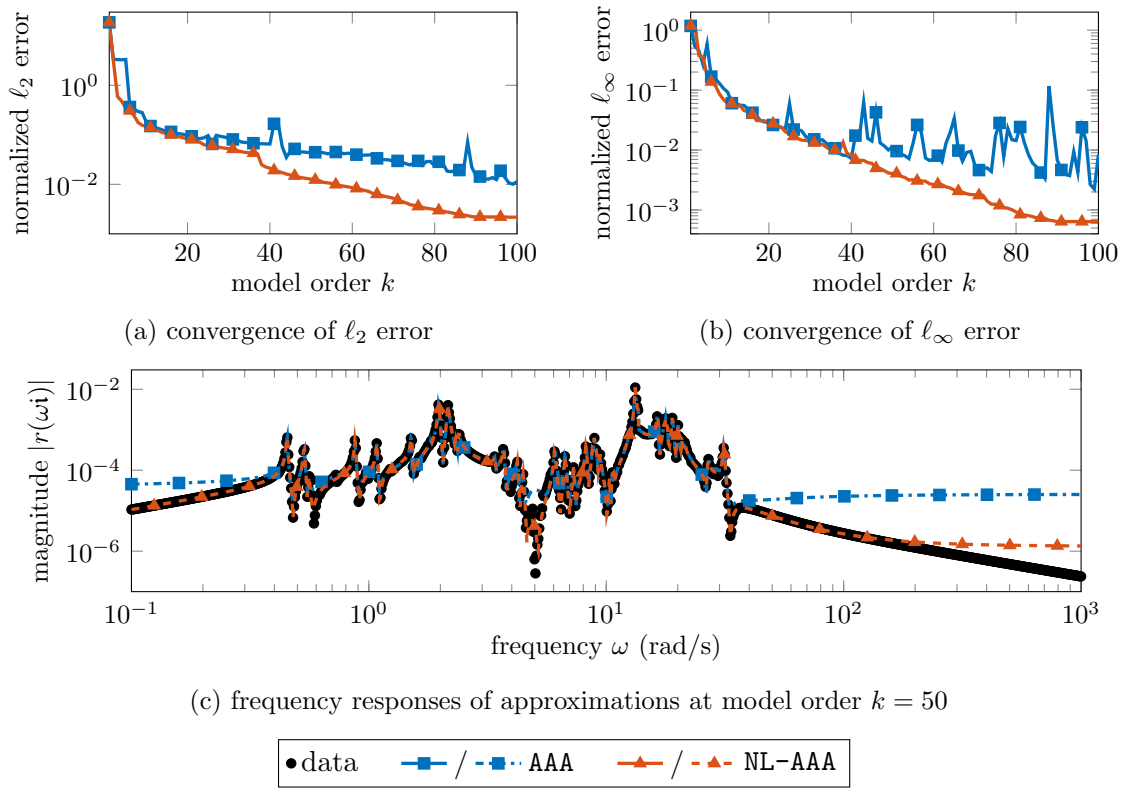
\begin{figure}[t]
  \centering
  \begin{subfigure}[b]{.49\linewidth}
    \centering
  \tikzexternalenable%
  \tikzsetnextfilename{ISS_conv_L2}%
  \begin{tikzpicture}[font = \plotfontsize]
  \pgfplotstableread{graphicsPaper/data/ISS_convergence_L2.csv}\tableERR

  \begin{semilogyaxis}[
    scale only axis,
    width              = .73\linewidth,
    height             = .4\linewidth,
    xmin               = 1,
    xmax               = 100,
    ymin               = 1e-3,
    ymax               = 3e+1,
    xminorticks        = false,
    yminorticks        = true,
    scaled x ticks     = false,
    xlabel             = {model order $k$},
    ylabel             = {normalized $\ltwo$ error},
    xlabel style       = {yshift = .3em},
    ylabel style       = {yshift = -.3em},
    x tick label style = {/pgf/number format/1000 sep={\,}},
    y tick label style = {/pgf/number format/1000 sep={\,}}
  ]
    \addplot[AAAConverge, mark repeat = {5}] table[x = iter, y = AAA]{\tableERR};
    \addplot[AAAAConverge, mark repeat = {5}] table[x = iter, y = AAAA]{\tableERR};
  \end{semilogyaxis}
\end{tikzpicture}%
  \tikzexternaldisable%

    \caption{convergence of $\ltwo$ error}
    \label{fig:ISS_convL2}
  \end{subfigure}%
  \hfill%
  \begin{subfigure}[b]{.49\linewidth}
    \centering
  \tikzexternalenable%
  \tikzsetnextfilename{ISS_conv_Linf}%
  \begin{tikzpicture}[font = \plotfontsize]
  \pgfplotstableread{graphicsPaper/data/ISS_convergence_Linf.csv}\tableERR

  \begin{semilogyaxis}[
    scale only axis,
    width              = .73\linewidth,
    height             = .4\linewidth,
    xmin               = 1,
    xmax               = 100,
    ymin               = 4e-4,
    ymax               = 2e+0,
    xminorticks        = false,
    yminorticks        = true,
    scaled x ticks     = false,
    xlabel             = {model order $k$},
    ylabel             = {normalized $\linf$ error},
    xlabel style       = {yshift = .3em},
    ylabel style       = {yshift = -.3em},
    x tick label style = {/pgf/number format/1000 sep={\,}},
    y tick label style = {/pgf/number format/1000 sep={\,}}
  ]
    \addplot[AAAConverge, mark repeat = {5}] table[x = iter, y = AAA]{\tableERR};
    \addplot[AAAAConverge, mark repeat = {5}] table[x = iter, y = AAAA]{\tableERR};
  \end{semilogyaxis}
\end{tikzpicture}%
  \tikzexternaldisable%

    \caption{convergence of $\linf$ error}
    \label{fig:ISS_conv_inf}
  \end{subfigure}

  \vspace{.5\baselineskip}
  \begin{subfigure}[b]{.98\linewidth}
    \centering
  \tikzexternalenable%
  \tikzsetnextfilename{ISS_freqPlot_wide}%
  \begin{tikzpicture}[font = \plotfontsize]
  \pgfplotstableread{graphicsPaper/data/ISS_frequencyResponse.csv}\tableRESP

  \begin{loglogaxis}[
    scale only axis,
    width              = .85\linewidth,
    height             = .2\linewidth,
    xmin               = 1e-1,
    xmax               = 1e+3,
    ymin               = 1e-7,
    ymax               = 3e-2,
    xminorticks        = true,
    yminorticks        = true,
    scaled x ticks     = false,
    clip mode          = individual,
    xlabel             = {frequency $\omega$ (rad/s)},
    ylabel             = {magnitude $\lvert r(\omega \imunit) \rvert$},
    xlabel style       = {yshift = .3em},
    ylabel style       = {yshift = -.3em},
    x tick label style = {/pgf/number format/1000 sep={\,}},
    y tick label style = {/pgf/number format/1000 sep={\,}}
  ]
    \addplot[trueData] table[x = mu, y = g]{\tableRESP};
    \addplot[AAAResponse] table[x = mu, y = AAA]{\tableRESP};
    \addplot[AAAAResponse] table[x = mu, y = AAAA]{\tableRESP};
  \end{loglogaxis}
\end{tikzpicture}%
  \tikzexternaldisable%

    \caption{frequency responses of approximations at model order $k = 50$}
    \label{fig:ISS_resp}
  \end{subfigure}
  
  \vspace{.5\baselineskip}
  \tikzexternalenable%
  \tikzsetnextfilename{LegendResults}%
  \begin{tikzpicture}
  \begin{axis}[%
    hide axis,
    scale only axis,
    width  = 1cm,
    height = 1cm,
    xmin   = 0,
    xmax   = 1,
    ymin   = 0,
    ymax   = 1,
    legend columns    = -1,
    legend cell align = {left},
    legend style      = {
      at     = {(0,0)},
      anchor = center,
      /tikz/every even column/.append style = {column sep = 0.0cm}}
  ]
    \addlegendimage{trueData} coordinates {(0, 0)};
    \addlegendentry{data\phantom{Pp}}

    \addlegendimage{AAAConverge}
    \addlegendentry{/}
    \addlegendimage{AAAResponse}
    \addlegendentry{\AAA{}\phantom{Pp}}

    \addlegendimage{AAAAConverge}
    \addlegendentry{/}
    \addlegendimage{AAAAResponse}
    \addlegendentry{\AAAA}
  \end{axis}
\end{tikzpicture}%
  \tikzexternaldisable%

  \caption{Convergence and approximation results of \AAA{} and \AAAA{}
    for the transfer function of the international space station model:
    The proposed \AAAA{} provides a monotonic convergence in the $\ltwo$
    error but also convinces in the $\linf$ error.
    For larger model orders, the error of the classical \AAA{} is visibly larger
    than for \AAAA{} in both error metrics.}
  \label{fig:ConvAndErr_ISS}
\end{figure}

Convergence of \AAA{} and \AAAA{} in the normalized $\ltwo$ norm
is shown in \Cref{fig:ISS_convL2}.
Due to its importance for the modeling of dynamical systems, we also show the
corresponding convergence in the $\linf$ norm in \Cref{fig:ISS_conv_inf}.
While initially as good as \AAAA{}, the classical \AAA{} algorithm diverges in accuracy
in both error metrics after model order $k = 40$.
This discrepancy increases for the remainder of the model orders displayed
until the final iteration $k = 100$, where the \AAAA{} approximation has
a normalized $\ltwo$ error of $0.0014$ while \AAA{} provides an error of
$0.0113$.
Thus, \AAAA{} performs one order of magnitude better in terms of accuracy for
the final approximation size.
As discussed in \Cref{sec:WFInits}, we see that \AAAA{} decreases the $\ltwo$
error monotonically in \Cref{fig:ISS_convL2}.
While the same claim does not hold for the $\linf$ norm, we see in
\Cref{fig:ISS_conv_inf} that the $\linf$ error of the \AAAA{} approximation
decreases monotonically nearly everywhere.

We emphasize that in model order reduction, the construction of highly accurate
models of smallest order is critical.
Furthermore, in many applications, a normalized $\ltwo$ error of about $1$\%
is considered as sufficiently accurate.
The proposed \AAAA{} achieves this target accuracy already at iteration
$k = 56$, while \AAA{} does not reach that accuracy even at iteration $100$.
In \Cref{fig:ISS_resp}, the frequency responses of the two approximations at
model order $k = 50$ are shown together with the sampled data.
We can see the qualitatively better fit provided by \AAAA{}.
While both approaches are able to accurately match the peaks in the data, the
data in between these peaks is visibly better matched by \AAAA{}.


\subsubsection{Vibrating plate}%
\label{sec:rom_plate}

As a second model order reduction example, we consider the vibrational response
of a strutted plate from~\cite{AumW23, ReiW25}.
This plate has been equipped with a set of tuned vibration absorbers, which
dampen the response of the plate around the frequency $48$\,Hz.
The vibrations are measured as root mean squared displacement of the internal
area of the plate in the vertical direction.
The data for this model are transfer function samples given at linearly
equidistant points in the interval $[1, 250] \cdot 2\pi \imunit$.
Since the transfer function of this model is a real-valued function on the
complete complex plane, it is not analytic in any open subset of the complex
plane.
We use the two algorithms, \AAA{} and \AAAA{}, to compute approximations to
the sampled data up to model order $60$.

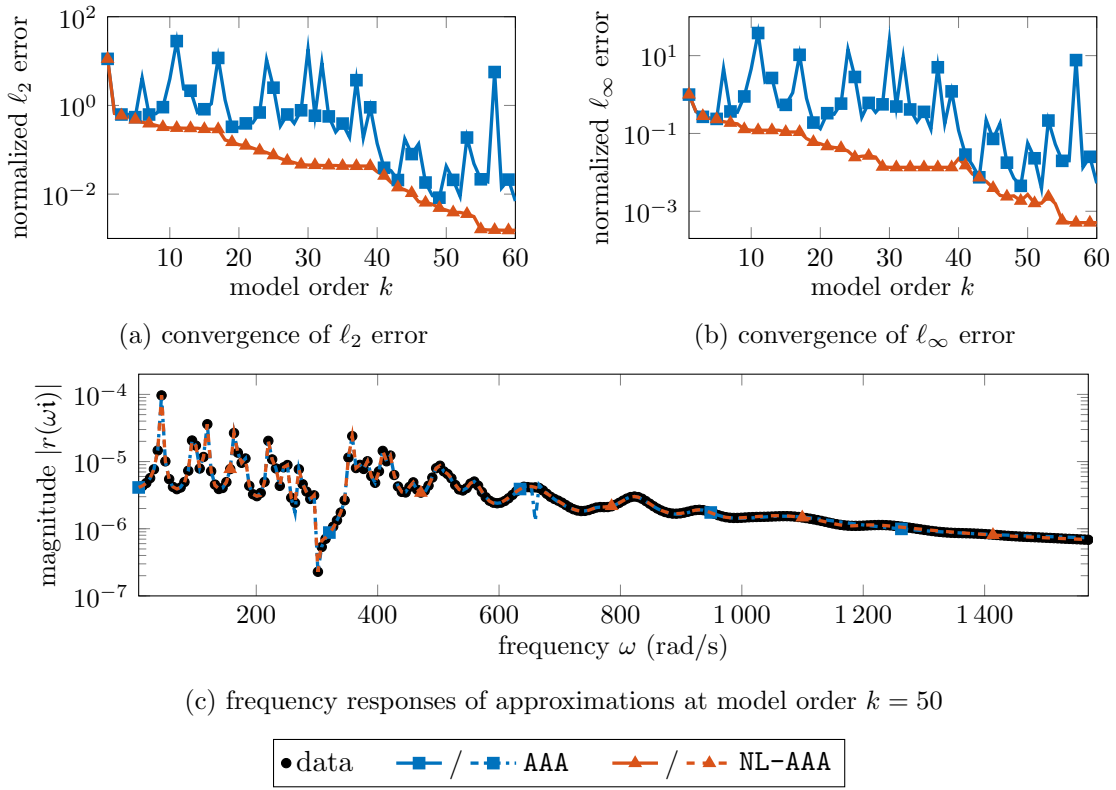
\begin{figure}[t]
  \centering
  \begin{subfigure}[b]{.49\linewidth}
    \centering
  \tikzexternalenable%
  \tikzsetnextfilename{VibePlate_conv_L2}%
  \begin{tikzpicture}[font = \plotfontsize]
  \pgfplotstableread{graphicsPaper/data/VibePlate_convergence_L2.csv}\tableERR

  \begin{semilogyaxis}[
    scale only axis,
    width              = .73\linewidth,
    height             = .4\linewidth,
    xmin               = 1,
    xmax               = 60,
    ymin               = 1e-3,
    ymax               = 1e+2,
    xminorticks        = false,
    yminorticks        = true,
    scaled x ticks     = false,
    xlabel             = {model order $k$},
    ylabel             = {normalized $\ltwo$ error},
    xlabel style       = {yshift = .3em},
    ylabel style       = {yshift = -.3em},
    x tick label style = {/pgf/number format/1000 sep={\,}},
    y tick label style = {/pgf/number format/1000 sep={\,}}
  ]
    \addplot[AAAConverge] table[x = iter, y = AAA]{\tableERR};
    \addplot[AAAAConverge] table[x = iter, y = AAAA]{\tableERR};
  \end{semilogyaxis}
\end{tikzpicture}%
  \tikzexternaldisable%

    \caption{convergence of $\ltwo$ error}
    \label{fig:VibePlate_convL2}
  \end{subfigure}%
  \hfill%
  \begin{subfigure}[b]{.49\linewidth}
    \centering
  \tikzexternalenable%
  \tikzsetnextfilename{VibePlate_conv_Linf}%
  \begin{tikzpicture}[font = \plotfontsize]
  \pgfplotstableread{graphicsPaper/data/VibePlate_convergence_Linf.csv}\tableERR

  \begin{semilogyaxis}[
    scale only axis,
    width              = .73\linewidth,
    height             = .4\linewidth,
    xmin               = 1,
    xmax               = 60,
    ymin               = 2e-4,
    ymax               = 1e+2,
    xminorticks        = false,
    yminorticks        = true,
    scaled x ticks     = false,
    xlabel             = {model order $k$},
    ylabel             = {normalized $\linf$ error},
    xlabel style       = {yshift = .3em},
    ylabel style       = {yshift = -.3em},
    x tick label style = {/pgf/number format/1000 sep={\,}},
    y tick label style = {/pgf/number format/1000 sep={\,}}
  ]
    \addplot[AAAConverge] table[x = iter, y = AAA]{\tableERR};
    \addplot[AAAAConverge] table[x = iter, y = AAAA]{\tableERR};
  \end{semilogyaxis}
\end{tikzpicture}%
  \tikzexternaldisable%

    \caption{convergence of $\linf$ error}
    \label{fig:VibePlate_convLinf}
  \end{subfigure}

  \vspace{.5\baselineskip}
  \begin{subfigure}[b]{.98\linewidth}
    \centering
  \tikzexternalenable%
  \tikzsetnextfilename{VibePlate_freqPlot_wide}%
  \begin{tikzpicture}[font = \plotfontsize]
  \pgfplotstableread{graphicsPaper/data/VibePlate_frequencyResponse.csv}\tableRESP

  \begin{semilogyaxis}[
    scale only axis,
    width              = .85\linewidth,
    height             = .2\linewidth,
    xmin               = 6.283185307179586,
    xmax               = 1.570796326794896e+3,
    ymin               = 1e-7,
    ymax               = 2e-4,
    xminorticks        = true,
    yminorticks        = true,
    scaled x ticks     = false,
    clip mode          = individual,
    xlabel             = {frequency $\omega$ (rad/s)},
    ylabel             = {magnitude $\lvert r(\omega \imunit) \rvert$},
    xlabel style       = {yshift = .3em},
    ylabel style       = {yshift = -.3em},
    x tick label style = {/pgf/number format/1000 sep={\,}},
    y tick label style = {/pgf/number format/1000 sep={\,}}
  ]
    \addplot[trueData] table[x = mu, y = g]{\tableRESP};
    \addplot[AAAResponse] table[x = mu, y = AAA]{\tableRESP};
    \addplot[AAAAResponse] table[x = mu, y = AAAA]{\tableRESP};
  \end{semilogyaxis}
\end{tikzpicture}%
  \tikzexternaldisable%

    \caption{frequency responses of approximations at model order $k = 50$}
    \label{fig:VibePlate_resp}
  \end{subfigure}
  
  \vspace{.5\baselineskip}
  \tikzexternalenable%
  \tikzsetnextfilename{LegendResults}%
  \begin{tikzpicture}
  \begin{axis}[%
    hide axis,
    scale only axis,
    width  = 1cm,
    height = 1cm,
    xmin   = 0,
    xmax   = 1,
    ymin   = 0,
    ymax   = 1,
    legend columns    = -1,
    legend cell align = {left},
    legend style      = {
      at     = {(0,0)},
      anchor = center,
      /tikz/every even column/.append style = {column sep = 0.0cm}}
  ]
    \addlegendimage{trueData} coordinates {(0, 0)};
    \addlegendentry{data\phantom{Pp}}

    \addlegendimage{AAAConverge}
    \addlegendentry{/}
    \addlegendimage{AAAResponse}
    \addlegendentry{\AAA{}\phantom{Pp}}

    \addlegendimage{AAAAConverge}
    \addlegendentry{/}
    \addlegendimage{AAAAResponse}
    \addlegendentry{\AAAA}
  \end{axis}
\end{tikzpicture}%
  \tikzexternaldisable%

  \caption{Convergence and approximation results of \AAA{} and \AAAA{}
    for the transfer function of the vibrating plate:
    The proposed \AAAA{} method provides a smooth and reliable decaying error
    behavior in both considered error metrics, while the classical \AAA{}
    has rapid changes in its approximation quality with an overall larger
    error.
    For model order $50$, the frequency response shows that \AAA{} introduces
    localized inaccuracies while \AAAA{} fits the data without visible
    differences.}
  \label{fig:ConvAndErr_VibePlate}
\end{figure}

The results of the computations can be seen in \Cref{fig:ConvAndErr_VibePlate}.
The normalized errors in \Cref{fig:VibePlate_convL2,fig:VibePlate_convLinf}
show the differences in the convergence behavior between \AAA{} and \AAAA{}.
As expected from the previous generic function examples, the classical \AAA{}
method struggles with the data coming from a
function with several non-analytic points compared to the previous model
reduction benchmark.
This can be seen in \Cref{fig:VibePlate_convL2,fig:VibePlate_convLinf} by the
large error deviations between iterations in either error metric.
In contrast, \AAAA{} provides the guaranteed monotonic error behavior in the
normalized $\ltwo$ norm and a similarly smooth and reliable error convergence in
the $\linf$ norm.
We can also see that overall the error of \AAAA{} is at least as good as \AAA{}
but significantly smaller for most approximation orders.

\Cref{fig:VibePlate_resp} shows the frequency response of the approximations
computed via \AAA{} and \AAAA{} for the model order $k = 50$ in comparison
to the given data.
For this example, we see that the larger error of the \AAA{} approximation stems
from localized inaccuracies close to the given data, e.g., at about
$630\imunit$,
there is a large, visible spike in the \AAA{} model that is not present in the
original data.
Similar to the examples in \Cref{sec:gradAnalysis}, we have observed that the
denominator of the \AAA{} approximation largely varies up to five order of
magnitude over the approximation range, which may explain the inaccuracies.
In contrast, we can see in \Cref{fig:VibePlate_resp} that the \AAAA{}
approximation accurately matches the given data.


\section{Conclusions}%
\label{sec:conclusions}

We introduced a new greedy approach for the construction of rational
approximations for given data.
Similar to the classical \AAA{} algorithm, the proposed method greedily selects
barycentric interpolation points based on an $\linf$ measure and subsequently
fits the barycentric weights to approximate the remaining data.
In contrast to the classical approach, we propose to use efficient refinement
procedures to solve the true rational least-squares problem to fit the
remaining data rather than fitting only a linearization.
The design of our algorithm is supported by the theoretical analysis of the
gradients used in the different minimization approaches, which provides a more
reliable and monotonic error behavior compared to \AAA{}.
This has been practically verified via numerical experiments using classical
problems from function approximation as well as model order reduction
benchmarks.


\section*{Acknowledgments}%
\addcontentsline{toc}{section}{Acknowledgments}

The work of Ackermann was supported in part by the
Simons Dissertation Fellowship in Mathematics. The works of Balicki and Gugercin were supported in part by 
US National Science Foundation grant DMS-2411141.


\addcontentsline{toc}{section}{References}
\bibliographystyle{plainurl}
\bibliography{bibtex/myref}

\begin{thebibliography}{10}

\bibitem{supAckBW26}
M.~S. Ackermann, L.~Balicki, and S.~W.~R. Werner.
\newblock Code, data and results for numerical experiments in ``{A} refined nonlinear least-squares method for the rational approximation problem'' (version 1.0), January 2026.
\newblock \href {https://doi.org/10.5281/zenodo.18317028} {\path{doi:10.5281/zenodo.18317028}}.

\bibitem{Ant05}
A.~C. Antoulas.
\newblock {\em Approximation of Large-Scale Dynamical Systems}, volume~6 of {\em Adv. Des. Control}.
\newblock SIAM, Philadelphia, PA, 2005.
\newblock \href {https://doi.org/10.1137/1.9780898718713} {\path{doi:10.1137/1.9780898718713}}.

\bibitem{AntA86}
A.~C. Antoulas and B.~D.~O. Anderson.
\newblock On the scalar rational interpolation problem.
\newblock {\em {IMA} J. Math. Control Inf.}, 3(2--3):61--88, 1986.
\newblock \href {https://doi.org/10.1093/imamci/3.2-3.61} {\path{doi:10.1093/imamci/3.2-3.61}}.

\bibitem{AntBG20}
A.~C. Antoulas, C.~A. Beattie, and S.~Gugercin.
\newblock {\em Interpolatory Methods for Model Reduction}.
\newblock Computational Science \& Engineering. SIAM, Philadelphia, PA, 2020.
\newblock \href {https://doi.org/10.1137/1.9781611976083} {\path{doi:10.1137/1.9781611976083}}.

\bibitem{AumW23}
Q.~Aumann and S.~W.~R. Werner.
\newblock Structured model order reduction for vibro-acoustic problems using interpolation and balancing methods.
\newblock {\em J. Sound Vib.}, 543:117363, 2023.
\newblock \href {https://doi.org/10.1016/j.jsv.2022.117363} {\path{doi:10.1016/j.jsv.2022.117363}}.

\bibitem{BenMS05}
P.~Benner, V.~Mehrmann, and D.~C. Sorensen.
\newblock {\em Dimension Reduction of Large-Scale Systems}, volume~45 of {\em Lect. Notes Comput. Sci. Eng.}
\newblock Springer, Berlin, Heidelberg, 2005.
\newblock \href {https://doi.org/10.1007/3-540-27909-1} {\path{doi:10.1007/3-540-27909-1}}.

\bibitem{BenSGetal21}
P.~Benner, W.~Schilders, S.~Grivet-Talocia, A.~Quarteroni, G.~Rozza, and L.~M. Silveira.
\newblock {\em Model Order Reduction. Volume 1: System- and Data-Driven Methods and Algorithms}.
\newblock De Gruyter, Berlin, Boston, 2021.
\newblock \href {https://doi.org/10.1515/9783110498967} {\path{doi:10.1515/9783110498967}}.

\bibitem{BerG15}
M.~Berljafa and S.~G{\"u}ttel.
\newblock Generalized rational {K}rylov decompositions with an application to rational approximation.
\newblock {\em {SIAM} J. Matrix Anal. Appl.}, 36(2):894--916, 2015.
\newblock \href {https://doi.org/10.1137/140998081} {\path{doi:10.1137/140998081}}.

\bibitem{BerG17}
M.~Berljafa and S.~G{\"u}ttel.
\newblock The {RKFIT} algorithm for nonlinear rational approximation.
\newblock {\em {SIAM} J. Sci. Comput.}, 39(5):A2049--A2071, 2017.
\newblock \href {https://doi.org/10.1137/15M1025426} {\path{doi:10.1137/15M1025426}}.

\bibitem{BerT04}
J.-P. Berrut and L.~N. Trefethen.
\newblock Barycentric {L}agrange interpolation.
\newblock {\em {SIAM} Rev.}, 46(3):501--517, 2004.
\newblock \href {https://doi.org/10.1137/S0036144502417715} {\path{doi:10.1137/S0036144502417715}}.

\bibitem{DesDA06}
D.~Deschrijver, T.~Dhaene, and G.~Antonini.
\newblock A convergence analysis of iterative macromodeling methods using {W}hitfield's estimator.
\newblock In {\em 2006 IEEE Workship on Signal Propagation on Interconnects}, pages 197--200, 2006.
\newblock \href {https://doi.org/10.1109/SPI.2006.289219} {\path{doi:10.1109/SPI.2006.289219}}.

\bibitem{DirMH23}
S.~Dirckx, K.~Meerbergen, and D.~Huybrechs.
\newblock {QR}-based parallel set-valued approximation with rational functions.
\newblock e-print 2312.10260, arXiv, 2023.
\newblock Numerical Analysis (math.NA).
\newblock \href {https://doi.org/10.48550/arXiv.2312.10260} {\path{doi:10.48550/arXiv.2312.10260}}.

\bibitem{DrmGB15}
Z.~Drma{\v{c}}, S.~Gugercin, and C.~Beattie.
\newblock Quadrature-based vector fitting for discretized $\mathcal{H}_{2}$ approximation.
\newblock {\em {SIAM} J. Sci. Comput.}, 37(2):A625--A652, 2015.
\newblock \href {https://doi.org/10.1137/140961511} {\path{doi:10.1137/140961511}}.

\bibitem{GosGB22}
I.~V. Gosea, S.~Gugercin, and C.~Beattie.
\newblock Data-driven balancing of linear dynamical systems.
\newblock {\em {SIAM} J. Sci. Comput.}, 44(1):A554--A582, 2022.
\newblock \href {https://doi.org/10.1137/21M1411081} {\path{doi:10.1137/21M1411081}}.

\bibitem{GugAB01}
S.~Gugercin, A.~C. Antoulas, and M.~Bedrossian.
\newblock Approximation of the international space station {1R} and {12A} models.
\newblock In {\em Proceedings of the 40th IEEE Conference on Decision and Control}, pages 1515--1516, 2001.
\newblock \href {https://doi.org/10.1109/CDC.2001.981109} {\path{doi:10.1109/CDC.2001.981109}}.

\bibitem{GusS99}
B.~Gustavsen and A.~Semlyen.
\newblock Rational approximation of frequency domain responses by vector fitting.
\newblock {\em {IEEE} Trans. Power Del.}, 14(3):1052--1061, 1999.
\newblock \href {https://doi.org/10.1109/61.772353} {\path{doi:10.1109/61.772353}}.

\bibitem{Hok20}
J.~M. Hokanson.
\newblock Multivariate rational approximation using a stabilized {S}anathanan-{K}oerner iteration.
\newblock e-print 2009.10803, arXiv, 2020.
\newblock Numerical Analysis (math.NA).
\newblock \href {https://doi.org/10.48550/arXiv.2009.10803} {\path{doi:10.48550/arXiv.2009.10803}}.

\bibitem{Ion13}
A.~C. Ionita.
\newblock {\em {L}agrange rational interpolation and its applications to approximation of large-scale dynamical systems}.
\newblock {D}issertation, Rice University, Houston, Texas, USA, 2013.
\newblock URL: \url{https://hdl.handle.net/1911/77180}.

\bibitem{Kre09}
K.~Kreutz-Delgado.
\newblock The complex gradient operator and the {CR}-calculus.
\newblock e-print 0906.4835, arXiv, 2009.
\newblock Optimization and Control (math.OC).
\newblock \href {https://doi.org/10.48550/arXiv.0906.4835} {\path{doi:10.48550/arXiv.0906.4835}}.

\bibitem{Lev59}
E.~C. Levy.
\newblock Complex-curve fitting.
\newblock {\em {IRE} Trans. Autom. Control}, AC-4(1):37--43, 1959.
\newblock \href {https://doi.org/10.1109/TAC.1959.6429401} {\path{doi:10.1109/TAC.1959.6429401}}.

\bibitem{MaE22}
A.~Ma and A.~Ege~Engin.
\newblock Orthogonal rational approximation of transfer functions for high-frequency circuits.
\newblock {\em Int. J. Circuit Theory Appl.}, 51(3):1007--1019, 2022.
\newblock \href {https://doi.org/10.1002/cta.3488} {\path{doi:10.1002/cta.3488}}.

\bibitem{MayA07}
A.~J. Mayo and A.~C. Antoulas.
\newblock A framework for the solution of the generalized realization problem.
\newblock {\em Linear Algebra Appl.}, 425(2--3):634--662, 2007.
\newblock Special issue in honor of P.~A. Fuhrmann, Edited by A.~C. Antoulas, U. Helmke, J. Rosenthal, V. Vinnikov, and E. Zerz.
\newblock \href {https://doi.org/10.1016/j.laa.2007.03.008} {\path{doi:10.1016/j.laa.2007.03.008}}.

\bibitem{morwiki_iss}
{MORwiki Benchmark Collection}.
\newblock Iss-12a.
\newblock hosted at {MORwiki} -- Model Order Reduction Wiki.
\newblock URL: \url{https://modelreduction.org/morwiki/International_Space_Station}.

\bibitem{NakST18}
Y.~Nakatsukasa, O.~S{\`e}te, and L.~N. Trefethen.
\newblock The {AAA} algorithm for rational approximation.
\newblock {\em {SIAM} J. Sci. Comput.}, 40(3):A1494--A1522, 2018.
\newblock \href {https://doi.org/10.1137/16M1106122} {\path{doi:10.1137/16M1106122}}.

\bibitem{NakT25}
Y.~Nakatsukasa and L.~N. Trefethen.
\newblock Applications of {AAA} rational approximation.
\newblock e-print 2510.16237, arXiv, 2025.
\newblock Numerical Analysis (math.NA).
\newblock \href {https://doi.org/10.48550/arXiv.2510.16237} {\path{doi:10.48550/arXiv.2510.16237}}.

\bibitem{ReiW25}
S.~Reiter and S.~W.~R. Werner.
\newblock Interpolatory model reduction of dynamical systems with root mean squared error.
\newblock {\em IFAC-Pap.}, 59(1):385--390, 2025.
\newblock 11th Vienna International Conference on Mathematical Modelling {MATHMOD} 2025.
\newblock \href {https://doi.org/10.1016/j.ifacol.2025.03.066} {\path{doi:10.1016/j.ifacol.2025.03.066}}.

\bibitem{SanK63}
C.~Sanathanan and J.~Koerner.
\newblock Transfer function synthesis as a ratio of two complex polynomials.
\newblock {\em {IEEE} Trans. Autom. Control}, 8(1):56--58, 1963.
\newblock \href {https://doi.org/10.1109/TAC.1963.1105517} {\path{doi:10.1109/TAC.1963.1105517}}.

\bibitem{Whi87}
A.~H. Whitfield.
\newblock Asymptotic behaviour of transfer function synthesis methods.
\newblock {\em Int. J. Control}, 45(3):1083--1092, 1987.
\newblock \href {https://doi.org/10.1080/00207178708933791} {\path{doi:10.1080/00207178708933791}}.

\bibitem{Wir27}
W.~Wirtinger.
\newblock Zur formalen {T}heorie der {F}unktionen von mehr komplexen {V}er{\"a}nderlichen.
\newblock {\em Math. Ann.}, 97(1):357--375, 1927.
\newblock \href {https://doi.org/10.1007/BF01447872} {\path{doi:10.1007/BF01447872}}.

\end{thebibliography}

\end{document}